\documentclass[12pt, letterpaper]{amsart}
\usepackage[utf8]{inputenc}
\usepackage[spanish]{babel}
\usepackage[all]{xy}
\usepackage{lipsum}
\usepackage{url}
\usepackage{tikz}
\usetikzlibrary{positioning,arrows.meta}
\usepackage{stackrel}
\usepackage{color}
\usepackage{imakeidx}
\makeindex
\usepackage{amsmath,amsthm,amssymb}
\usepackage{xcolor}

\usepackage[left=1in,right=1in,bottom=0.5in,top=0.8in]{geometry}
\usepackage{amsfonts}

\usepackage{graphicx}
\usepackage[font=small,labelfont=bf]{caption}
\usepackage{epstopdf}
\usepackage[pdfpagelabels,hyperindex]{hyperref}

\usepackage{float}
\usepackage{pgfplots}
\usepackage{listings}
\usepackage{longtable}
\usepackage{mathrsfs}

\newcommand{\aspas}[1]{``{#1}''}

\theoremstyle{definition}
\newtheorem{theorem}{Teorema}[section]

\newtheorem{example}[theorem]{Ejemplo}
\newtheorem{proposition}[theorem]{Proposición}
\newtheorem{definition}[theorem]{Definición}
\newtheorem{conjecture}[theorem]{Conjetura}
\newtheorem{remark}[theorem]{Observación}
\newtheorem{corollary}[theorem]{Corolario}

\newtheorem{problem}[theorem]{Problema}

\numberwithin{equation}{section}

\begin{document}

%%%%%%%%%%%%%%%%%%%%%%%%%%%%%%%%%%%%%%%%%%%%%%%%%%%%%%%%%%%%

\vspace{0.5in}

\renewcommand{\bf}{\bfseries}
\renewcommand{\sc}{\scshape}
%insert defs/styles
\vspace{0.5in}

\title[Categoría de una aplicación y análisis no lineal]%
{Categoría de una aplicación y análisis no lineal}

%    Information for first author:
\author{Cesar A. Ipanaque Zapata}
\address{Departamento de Matem\'atica, IME Universidade de S\~ao Paulo\\
Rua do Mat\~ao 1010 CEP: 05508-090 S\~ao Paulo-SP, Brazil}
%    Current address (if needed):
%\curraddr{}
\email{cesarzapata@usp.br}
%\thanks{}

%    Information for second author (if needed):
%\author{}
%\address{}
%\email{}
%\thanks{Support information for the second author.}

%    General info
%%%%%%%%%%%%%%%%%%%%%%%%%%%%%%%%%%%%%%%%%%%%%%%%%%%
%                                                                                                                           %
%         Please use the current 2010 Mathematics Subject Classification:             %
%         http://www.ams.org/mathscinet/msc/                                                        %
%         http://www.zentralblatt-math.org/msc/en/    
%55M30  	Lyusternik-Shnirel'man category of a space, topological complexity à la Farber, topological robotics (topological aspects)%
%55P10  	Homotopy equivalences in algebraic topology
%65H20  	Global methods, including homotopy approaches to the numerical solution of nonlinear equations
%58B05  	Homotopy and topological questions for infinite-dimensional manifolds
%55Q05  	Homotopy groups, general; sets of homotopy classes
%%%%%%%%%%%%%%%%%%%%%%%%%%%%%%%%%%%%%%%%%%%%%%%%%%%

%\thanks {The author wishes to acknowledge support for this research, from FAPESP 2016/18714-8.}

\begin{abstract} Un problema clásico en análisis es resolver ecuaciones no lineales de la forma \begin{equation*}
      F(x)=0,
  \end{equation*} donde $F:D^n\to \mathbb{R}^m$ es una aplicación continua del disco unitario cerrado $D^n\subset\mathbb{R}^n$ en $\mathbb{R}^m$. Una técnica topológica, que existe en la literatura, para la existencia soluciones de ecuaciones no lineales es la teoría del grado topológico. En este trabajo, usaremos la teoría de categoría de una aplicación para resolver el problema de existencia de soluciones de ecuaciones no lineales. Esta teoría, como mostraremos en este trabajo, da una técnica topológica alternativa para estudiar ecuaciones no lineales.
\end{abstract}

\maketitle
%%%%%%%%%%%%%%%%%%%%% Titulo del articulo en ingles%%%%%%%%%
\vspace{5cm}
\begin{center}
\noindent{\textbf{Category of a map and nonlinear analysis}}
\end{center}
\begin{quote}
\textbf{Abstract:}  A classic problem in analysis is to solve nonlinear equations of the form \begin{equation*}
        F(x)=0,
    \end{equation*} where $F:D^n\to \mathbb{R}^m$ is a continuous map of the closed unit disk $D^n\subset\mathbb{R}^n$ in $\mathbb{R}^m$. A topological technique, which exists in the literature, for the existence of solutions of nonlinear equations is the topological degree theory. In this work, we will use the category of a map theory to solve the problem of existence of solutions of nonlinear equations. This theory, as we will show in this work, provides an alternative topological technique to study nonlinear equations.

\vspace{0.3cm}
\end{quote}
\tableofcontents

%%%%%%%%%%%%%%%%%%%%%%%%%%%%%%%%%%%%%%%%%%%%%%%%%%%%%%%%%%%%%%

\section{Introducción}
Un problema clásico en análisis es resolver ecuaciones no lineales de la forma \begin{equation*}
      F(x)=0,
  \end{equation*} donde $F:D^n\to \mathbb{R}^m$ es una aplicación continua del disco unitario cerrado $D^n\subset\mathbb{R}^n$ en $\mathbb{R}^m$. Note que, si existe $x\in \partial D^n=S^{n-1}$ tal que $F(x)=0$, tenemos inmediatamente que tal ecuación admite solución. Así que, en adelante vamos a suponer que $F(x)\neq 0$, para cualquier $x\in \partial D^n=S^{n-1}$ y consideraremos $F_\mid:S^{n-1}\to\mathbb{R}^m-\{0\}$ como la aplicación restricción. Una técnica topológica, que existe en la literatura, para la existencia de ecuaciones no lineales es la teoría del grado topológico. En este trabajo, usaremos la teoría de categoría de una aplicación para resolver el problema de existencia de soluciones de ecuaciones no lineales. Esta teoría, como mostraremos en este trabajo, da una técnica alternativa para estudiar ecuaciones no lineales.   

\medskip En el Capítulo~\ref{chap:preliminares} daremos una revisión de Homotopía. En la Sección~\ref{sec:homotopía} presentamos la definición y propiedades básicas de la relación de homotopía entre aplicaciones. La relación de homotopía es una relación de equivalencia (Proposición~\ref{proposition:relacion-homotopia}). Proposición~\ref{proposition:homotopia-compuesta} establece que la relación de homotopía se mantiene con respecto a la operación de composición. En la Sección~\ref{sec:nulhomotopy} recordamos la definición y propiedades básicas de nulo homotopía y contráctibilidad. La Proposición~\ref{proposition:o-composta} dice que la composición de dos aplicaciones es nulo homotópica siempre que una de las aplicaciones sea nulo homotópica. El caso particular cuando la identidad es nulo homotópica origina el concepto de contractibilidad (Definición~\ref{definition:contractil-space}). Proposición~\ref{proposition:o-contractil-nulo} muestra que  toda aplicación cuyo dominio o codominio sea contráctil es nulo homotópica. Además, la Proposición~\ref{proposition:codomain-contractil} dice que toda aplicación continua con codominio contráctil es homotópica a una aplicación constante deseada. En particular, cualesquier dos aplicaciones con el mismo codominio contráctil son homotópicas (ver Corollary~\ref{corollary:contractil-codo-maps-homotopic}). La Proposición~\ref{proposition:domain-contractil} afirma que toda aplicación continua con dominio contráctil y contradominio conexo por caminos es homotópica a una aplicación constante deseada. Como consecuencia, cualesquier dos aplicaciones con el mismo dominio contráctil y el mismo codominio conexo por caminos son homotópicas (Corolario~\ref{corollary:contractible-domain-path-conn-codomain-homotopy}). En la Sección~\ref{sec:retract} presentamos el concepto y propiedades básicas de retracto y retracto por deformación.  Proposición~\ref{proposition:retract-contractible-is-contractible} dice que todo retracto de un espacio contráctil es contráctil. En particular, tenemos que la esfera unitaria $S^{n-1}$ no es un retracto del disco unitario $D^n$ (Proposición~\ref{proposition:esfera-no-retracto}). Como consecuencia, se establece el famoso teorema del punto fijo de Brouwer (Proposición~\ref{proposition:brouwer-fixed-point-theorem}). Proposición~\ref{proposition:retract-extension} presenta una caracterización de retractos en términos de extensión de aplicaciones. Presentamos también el concepto de equivalencia homotópica (Definición~\ref{definition:homotopy-equivalence}). %Para complementar este capítulo se presenta una lista de ejercicios en la Sección~\ref{sec:exer-homotopia}.

\medskip En el Capítulo~\ref{chap:categoria-aplicacion} daremos una revisión de categoría de una aplicación. En la Sección~\ref{sec:definiciones-cat} presentamos la definición y propiedades básicas de la categoría de una aplicación. Una propiedad es que la categoría de una aplicación es un invariante homotópico (Proposición~\ref{proposition:cat-aplicacao-invariante-homotopico}). La Proposición~\ref{proposition:cat-composta} muestra el comportamiento de la categoría con respecto a la composición. Además, Proposición~\ref{proposition:cat-composta-equivalencia-homotopica} dice que si componemos a una aplicación con una equivalencia homotópica su categoría no se altera. Se cumple que la categoría de una aplicación se puede realizar como la categoría de una fibración (ver Proposición~\ref{proposition:cat-map-cat-fibration}). En la Sección~\ref{sec:categoriaLS} recordamos la definición y propiedades básicas de categoría LS. La categoría LS da una cota superior para la categoría de una aplicación (Proposición~\ref{proposition:cat-catLS}). En particular, se cumple que la categoría de una aplicación con dominio o codominio una esfera es a lo máximo 2 (Corolario~\ref{corollary:cat-do-codo-esfera}). La Proposición~\ref{proposition:cat-no-nulo-2} da una caracteización de cuando una aplicación continua, con dominio o codominio una esfera, no es nulo homotópica. El comportamiento de la categoría LS para los retractos está dada en la Proposición~\ref{proposition:x-domina-y-catLS}. En particular se muestra que la categoría LS es un invariante homotópico (Corolario~\ref{corollary:cat-space-invariante}). Además, se tiene que la categoría LS de un espacio siempre le gana a la categoría LS de sus retractos (Corolario~\ref{corollary:catLS-retractos}). La Proposición~\ref{proposition:cat-map-inversa-homotopica} muestra la categoría de una equivalencia homotópica. En la Definición~\ref{definition:secat-def} recordamos la noción de categoría seccional. Una conexión entre categoría, categoría seccional y categoría LS está dada en el Teorema~\ref{theorem:cat-secat-maior-cat}. La Proposición~\ref{proposition:secat-1-cat-igual-catLS} muestra una igualdad entre la categoría de una aplicación, con categoría seccional uno, y la categoría LS de su codomoinio. En particular, se tiene que la categoría de la proyección sobre la primera coordenada definida sobre el espacio de configuraciones ordenado con dos puntos coincide con la categoría del espacio, siempre que tal espacio admita una aplicación continua sin punto fijo (Proposición~\ref{proposition:cat-proyection-section}). %Para complementar este capítulo se presenta una lista de ejercicios en la Sección~\ref{sec:exercise-categoria-aplicacao}. 

\medskip En el Capítulo~\ref{chap:aplicaciones-analise} usaremos la teoría de categoría de una aplicación para estudiar la existencia de soluciones de ecuaciones no lineales. En la Sección~\ref{sec:ecuaciones-no-lineales} se presenta el problema clásico de análisis no lineal, en particular de la existencia de soluciones, el cual pretendemos resolver usando la teoría de categoría de una aplicación. Una conexión entre la existencia de soluciones de una ecuación no lineal y la categoría de una aplicación está dada en el Teorema~\ref{theorem:extension-cero-cat-2}. En particular, el  Corolario~\ref{corollary:cat-2-solucion} presenta una condición en términos de categoría para la existencia de soluciones de una ecuación no lineal. Proposición~\ref{proposition:centro-radio-cualquer} dice que podemos considerar cualquier disco cerrado de cualquier radio y centrado en cualquier punto. Una condición en términos de categoría seccional para la existencia de soluciones está dada en la Proposición~\ref{proposition:secat-existencia-soluciones}. En la Sección~\ref{sec:ejemplos-concretos} presentamos una serie de ejemplos del uso de categoría en la existencia de soluciones de ecuaciones no lineales. %Para complementar este capítulo se presenta una lista de ejercicios en la Sección~\ref{sec:exercises-ecuaciones-no-lineales}.  

\medskip Adicionalmente, en cada Capítulo se proponen problemas que pueden ser consideramos como proyectos de investigación a futuro. 
%%%%%%%%%%%%%%%%%%%%%%%%%%%%%%%%%%%%%%%%%%%%%%%%%%%%%%%%%%%%%%%%%%%%%%%%%%%%5

\section{Preliminares}\label{chap:preliminares}
En este capítulo daremos una revisión de Homotopía. En la Sección~\ref{sec:homotopía} presentamos la definición y propiedades básicas de la relación de homotopía entre aplicaciones. En la Sección~\ref{sec:nulhomotopy} recordamos la definición y propiedades básicas de nulo homotopía y contráctibilidad. En la Sección~\ref{sec:retract} presentamos el concepto y propiedades básicas de retracto y retracto por deformación. 

\medskip Tendremos en cuenta que una aplicación es una función $f:X\to Y$ tal que $\mathrm{Dom}(f)=X$. 

\subsection{Homotopía}\label{sec:homotopía}
En esta sección presentamos la definición y propiedades básicas de la relación de homotopía entre aplicaciones. La relación de homotopía es una relación de equivalencia (Proposición~\ref{proposition:relacion-homotopia}). Proposición~\ref{proposition:homotopia-compuesta} establece que la relación de homotopía se mantiene con respecto a la operación de composición. 

\medskip Note que, toda aplicación de la forma $H:X\times Z\to Y$ define una familia de aplicaciones $\{H_z:X\to Y\}_{z\in Z}$, donde cada $H_z:X\to Y$ está dada por $H_z(x)=H(x,z)$, para todo $x\in X$. Definiremos el concepto de homotopía el cuál será esencial en este trabajo. 

\begin{definition}[Homotopía]\label{definition:homotopia}
    Sean $f,g:X\to Y$ aplicaciones continuas.
    \begin{enumerate}
        \item Una \textit{homotopía}\index{Homotopía} de $f$ en $g$ es una aplicación continua $H:X\times [0,1]\to Y$ tal que $H_0=f$ y $H_1=g$.  
        \item Diremos que $f$ es \textit{homotópica}\index{Homotópica} a $g$, denotado por $f\simeq g$, cuando existe una homotopía de $f$ en $g$.  
    \end{enumerate}
\end{definition}

\begin{example}\label{first-example}
\noindent\begin{enumerate}
    \item[(1)] Sean $X$ un espacio topológico y $f,g:X\to\mathbb{R}^n$ aplicaciones continuas. Tenemos que $f\simeq g$. De hecho, consideremos la aplicación $H:X\times [0,1]\to \mathbb{R}^n$ dada por $H(x,t)=(1-t)f(x)+tg(x)$, para todo $(x,t)\in X\times [0,1]$. Note que, $H$ es continua (ya que $\mathbb{R}^n$ es un $\mathbb{R}$-espacio vectorial topológico). Además, $H(x,0)=f(x), \forall x\in X$, y $H(x,1)=g(x), \forall x\in X$.  
    \item[(2)] Sean $f,g:[0,1]\to S^1$ aplicaciones continuas dadas por $f(s)=\left(\cos{2\pi s},\sin{2\pi s}\right), \forall s\in [0,1]$, y $g(s)=(1,0), \forall s\in [0,1]$. Tenemos que $f\simeq g$. De hecho, consideremos la aplicación $H:[0,1]\times [0,1]\to S^1$ dada por $H(s,t)=f((1-t)s), \forall (s,t)\in [0,1]\times [0,1]$. Note que, $H$ está bien definida (ya que $(1-t)s\in [0,1]$), es continua (ya que es composición de aplicaciones continuas) y $H(s,0)=f(s), \forall s\in [0,1]$, $H(s,1)=f(0)=(1,0), \forall s\in [0,1]$.
\end{enumerate}    
\end{example}

Veamos que la relación \aspas{$\simeq$} es una relación de equivalencia en el conjunto $\mathrm{Map}(X,Y)$ de todas las aplicaciones continuas de $X$ en $Y$.

\begin{proposition}\label{proposition:relacion-homotopia}
    La relación \aspas{$\simeq$} es una relación de equivalencia en el conjunto $\mathrm{Map}(X,Y)$, o sea, cumple las siguientes propiedades: \begin{enumerate}
        \item $f\simeq f$ para todo $f\in \mathrm{Map}(X,Y)$ (Propiedad reflexiva).
        \item Para $f,g\in \mathrm{Map}(X,Y)$. Si $f\simeq g$ entonces $g\simeq f$ (Propiedad simétrica).
        \item Para $f,g,h\in \mathrm{Map}(X,Y)$. Si $f\simeq g$ y $g\simeq h$ entonces $f\simeq h$ (Propiedad transitiva).
    \end{enumerate}
\end{proposition}
\begin{proof}
    \noindent\begin{enumerate}
        \item Sea $f\in \mathrm{Map}(X,Y)$. Definamos $H:X\times [0,1]\to Y$ por $H(x,t)=f(x)$ para todo $(x,t)\in X\times [0,1]$. Note que, $H$ es una homotopía de $f$ en $f$. Así, $f\simeq f$. 
        \item Sean $f,g\in \mathrm{Map}(X,Y)$ tal que $f\simeq g$. Sea $H:X\times [0,1]\to Y$ una homotopía de $f$ en $g$. Definamos $\overleftarrow{H}:X\times [0,1]\to Y$ por $\overleftarrow{H}(x,t)=H(x,1-t)$ para todo $(x,t)\in X\times [0,1]$. Note que $\overleftarrow{H}$ es una homotopía de $g$ en $f$. Así, $g\simeq f$.
        \item Sean $f,g,h\in \mathrm{Map}(X,Y)$ tales que $f\simeq g$ y $g\simeq h$. Sean $F:X\times [0,1]\to Y$ una homotopía de $f$ en $g$ y $G:X\times [0,1]\to Y$ una homotopía de $g$ en $h$. Definamos $H:X\times [0,1]\to Y$ por \[H(x,t)=\begin{cases}
            F(x,2t),& \hbox{si $0\leq t\leq 1/2$;}\\
            G(x,2t-1),& \hbox{si $1/2\leq t\leq 1$.}\\
        \end{cases}\] Note que, $H$ es una homotopía de $f$ en $h$. Así, $f\simeq h$.
    \end{enumerate}
\end{proof}

Por la Proposición~\ref{proposition:relacion-homotopia} podemos considerar el conjunto cociente \begin{align*}
    \mathrm{Map}(X,Y)/\simeq=\{[f]:~f\in \mathrm{Map}(X,Y)\},
\end{align*} donde cada clase de equivalencia $[f]=\{g\in\mathrm{Map}(X,Y):g\simeq f\}$ es llamada \textit{clase de homotopía}\index{Clase de homotopía} de $f$. Así, el conjunto cociente $\mathrm{Map}(X,Y)/\simeq$ es llamado \textit{conjunto de clases de homotopía}\index{Conjunto de clases de homotopía} de $X$ en $Y$. 

\begin{example}\label{mas-ejemplos}
    \noindent\begin{enumerate}
        \item[(1)] Sean $f,g:[0,1]\to S^1$ aplicaciones continuas dadas por $f(s)=\left(\cos{2\pi s},\sin{2\pi s}\right), \forall s\in [0,1]$, y $g(s)=-f(s), \forall s\in [0,1]$. Tenemos que $f\simeq g$. De hecho, por el Ítem (2) del Ejemplo~\ref{first-example} tenemos que $f\simeq\overline{(1,0)}$ y $g\simeq\overline{(-1,0)}$. Veamos ahora que $\overline{(1,0)}\simeq \overline{(-1,0)}$. Consideremos los caminos $\gamma_1,\gamma_2:[0,1]\to S^1$ dados por \begin{align*}
          \gamma_1(s)&=\dfrac{(1-s)(1,0)+s(0,1)}{\|(1-s)(1,0)+s(0,1)\|},\\  
          \gamma_2(s)&=\dfrac{(1-s)(0,1)+s(-1,0)}{\|(1-s)(0,1)+s(-1,0)\|}.\\  
        \end{align*} El camino $\gamma:[0,1]\to S^1$ dado por $$\gamma(s)=\begin{cases}
            \gamma_1(2s),&\hbox{si $0\leq s\leq 1/2$;}\\
            \gamma_2(2s-1),&\hbox{si $1/2\leq s\leq 1$}\\
        \end{cases}$$ cumple que $\gamma(0)=(1,0)$ y $\gamma(1)=(-1,0)$. Luego la aplicación $H:[0,1]\times [0,1]\to S^1$ dada por $H(s,t)=\gamma(t), \forall (s,t)\in [0,1]\times [0,1]$ es una homotopía de $\overline{(1,0)}$ en $\overline{(-1,0)}$. Así, por la Proposición~\ref{proposition:relacion-homotopia} podemos concluir que $f\simeq g$.
        \item[(2)] Sean $f,g:X\to \mathbb{R}^n\setminus\{0\}$ aplicaciones continuas tales que para cada $x\in X$ existe $\lambda> 0$ con $g(x)=\lambda f(x)$. Tenemos que $f\simeq g$. De hecho, consideremos la aplicación $H:X\times [0,1]\to \mathbb{R}^n\setminus\{0\}$ dada por $H(x,t)=(1-t)f(x)+tg(x), \forall (x,t)\in X\times [0,1]$. Veamos que $H$ está bien definida, o sea, $H(x,t)\in \mathbb{R}^n\setminus\{0\}$, para todo $(x,t)\in X\times [0,1]$. Supongamos que $(1-t)f(x)+tg(x)=0$ para algún $(x,t)\in X\times [0,1]$. Note que $0<t<1$. Para tal $x\in X$, consideremos $\lambda> 0$ con $g(x)=\lambda f(x)$. Luego, \begin{align*}
            (1-t)f(x)+t\lambda f(x)&=(0,0),\\
            (1-t)f(x)&=-t\lambda f(x),\\
            (1-t)\|f(x)\|&=t\lambda \|f(x)\|,\\
            1-t&=t\lambda.
        \end{align*} Entonces, reemplazando en la igualdad $(1-t)f(x)+t\lambda f(x)=0$, obtenemos que $2(1-t)f(x)=0$. Como $f(x)\neq 0$ entonces $t=1$, lo cual es una contradicción. Así, $H$ está bien definida. Luego, podemos concluir que $H$ es una homotopía de $f$ en $g$.  
    \end{enumerate}
\end{example}

\begin{example}
    \noindent\begin{enumerate}
        \item[(1)] Sean $f,g:S^1\to \mathbb{R}^2\setminus\{(0,0)\}$ aplicaciones continuas dadas por $f(x,y)=(x,y), \forall (x,y)\in S^1$, y $g(x,y)=(x,y)+(3,3), \forall (x,y)\in S^1$. Note que, $f$ no es homotópica a $g$, ya que si $H:S^1\times [0,1]\to \mathbb{R}^2\setminus\{(0,0)\}$ es una homotopía de $f$ en $g$ entonces existen $(x_0,t_0)\in S^1\times [0,1]$ tal que $H(x_0,t_0)=(0,0)$, lo cual es una contradicción.
    \item[(2)] Sean $f,g:S^1\to \mathbb{R}^3\setminus\{(0,0,0)\}$ aplicaciones continuas dadas por $f(x,y)=(x,y,0), \forall (x,y)\in S^1$, y $g(x,y)=(x,y,0)+(3,3,0), \forall (x,y)\in S^1$. Veamos que, $f\simeq g$. La aplicación $F:S^1\times [0,1]\to \mathbb{R}^3\setminus\{(0,0,0)\},~H((x,y),t)=(x,y,t)$, es una homotopía de $f$ en $h$, donde $h(x,y)=(x,y,1)$. Además, la aplicación $G:S^1\times [0,1]\to \mathbb{R}^3\setminus\{(0,0,0)\},~G((x,y),t)=(1-t)(x,y,1)+tg(x,y)$, es una homotopía de $h$ en $g$. Así, por la transitividad de la relación de homotopía (ver Proposición~\ref{proposition:relacion-homotopia}), obtenemos que $f\simeq g$.  
    \end{enumerate}
\end{example}

\medskip Ahora, veamos que la relación de homotopía se mantiene con respecto a la operación de composición.

\begin{proposition}\label{proposition:homotopia-compuesta}
Sean $f,g:X\to Y$ y $f',g':Y\to Z$ aplicaciones continuas. Si $f\simeq g$ y $f'\simeq g'$ entonces $(f'\circ f)\simeq (g'\circ g)$.
\end{proposition}
\begin{proof}
    Sea $H:X\times [0,1]\to Y$ una homotopía de $f$ en $g$ y $G:Y\times [0,1]\to Z$ una homotopía de $f'$ en $g'$. Definamos $F:X\times [0,1]\to Z$ por \[F(x,t)=G(H(x,t),t), \forall (x,t)\in X\times [0,1].\] Note que, $F$ es una homotopía de $f'\circ f$ en $g'\circ g$. Así, $(f'\circ f)\simeq (g'\circ g)$.
\end{proof}

\begin{corollary}\label{corollary:homo-igual}
  Sean $f,g:X\to Y$, $h':W\to X$ y $h:Y\to Z$ aplicaciones continuas. Si $f\simeq g$ entonces $f\circ h'\simeq g\circ h'$ y $h\circ f\simeq h\circ g$.   
\end{corollary}

\subsection{Nulo homotopía y espacio contráctil}\label{sec:nulhomotopy}
En esta sección recordamos la definición y propiedades básicas de nulo homotopía y contráctibilidad. La Proposición~\ref{proposition:o-composta} dice que la composición de dos aplicaciones es nulo homotópica siempre que una de las aplicaciones sea nulo homotópica. El caso particular cuando la identidad es nulo homotópica origina el concepto de contractibilidad (Definición~\ref{definition:contractil-space}). Proposición~\ref{proposition:o-contractil-nulo} muestra que  toda aplicación cuyo dominio o codominio sea contráctil es nulo homotópica. Además, la Proposición~\ref{proposition:codomain-contractil} dice que toda aplicación continua con codominio contráctil es homotópica a una aplicación constante deseada. En particular, cualesquier dos aplicaciones con el mismo codominio contráctil son homotópicas (ver Corollary~\ref{corollary:contractil-codo-maps-homotopic}). La Proposición~\ref{proposition:domain-contractil} afirma que toda aplicación continua con dominio contráctil y contradominio conexo por caminos es homotópica a una aplicación constante deseada. Como consecuencia, cualesquier dos aplicaciones con el mismo dominio contráctil y el mismo codominio conexo por caminos son homotópicas (Corolario~\ref{corollary:contractible-domain-path-conn-codomain-homotopy}).      

\begin{definition}[Nulo homotopía]
    Una aplicación continua $f:X\to Y$ es llamada \textit{nulo homotópica}\index{Nulo homotópica} cuando existe una constante $y_0\in Y$ y una homotopía de $f$ en la aplicación constante $\overline{y_0}:X\to Y,~x\mapsto y_0$. O sea, existe una aplicación continua $H:X\times [0,1]\to Y$ tal que $H_0=f$ y $H_1=\overline{y_0}$ (para algún $y_0\in Y$). Tal $H$ es llamada \textit{nulo homotopía}\index{Nulo homotopía}.  
\end{definition}

\begin{example}
    Toda aplicación constante es nulo homotópica. 
\end{example}

El siguiente resultado muestra que la composición de dos aplicaciones es nulo homotópica siempre que una de las aplicaciones sea nulo homotópica.   

\begin{proposition}\label{proposition:o-composta}
 Sean $f:X\to Y$, $g:Y\to Z$ aplicaciones continuas. Si $f$ o $g$ es nulo homotópica entonces la composición $g\circ f$ es nulo homotópica.   
\end{proposition}
\begin{proof}
    Supongamos que $f\simeq \overline{y_0}$ (para algún $y_0\in Y$). Por el Corolario~\ref{corollary:homo-igual}, $g\circ f\simeq g\circ\overline{y_0}$. Pero, $g\circ\overline{y_0}=\overline{z_0}$, donde $z_0=g(y_0)$. Así, $g\circ f$ es nulo homotópica.  

    Ahora, supongamos que $g\simeq\overline{z_0}$ (para algún $z_0\in Z$), aquí la aplicación constante $\overline{z_0}$ es de $Y$ en $Z$. Nuevamente por el Corolario~\ref{corollary:homo-igual}, $g\circ f\simeq \overline{z_0}\circ f$. Pero, $\overline{z_0}\circ f$ es la aplicación constante $\overline{z_0}$ de $X$ en $Z$. Así, $g\circ f$ es nulo homotópica.   
\end{proof}

Para un conjunto $X$, denotemos por $1_X:X\to X,~x\mapsto x$, a la aplicación identidad. El caso particular cuando la identidad es nulo homotópica origina el concepto de contractibilidad.  

\begin{definition}[Contráctil]\label{definition:contractil-space}
\noindent\begin{enumerate}
    \item Un espacio topológico $X$ es llamado \textit{contráctil}\index{Contráctil}  cuando la aplicación identidad $1_X:X\to X$ es nulo homotópica.  O sea, existe una aplicación continua $H:X\times [0,1]\to X$ tal que $H_0=1_X$ y $H_1=\overline{x_0}$ (para algún $x_0\in X$).   
    \item  Sea $A\subset X$. Diremos que $A$ es \textit{contráctil en}\index{Contráctil en} $X$ cuando la aplicación inclusión $U\hookrightarrow X$ es nulo homotópica. O sea, existe una aplicación continua $H:U\times [0,1]\to X$ tal que $H(x,0)=x$ para todo $x\in U$ y $H(x,1)=x_0$ para todo $x\in X$ (para algún $x_0\in X$).   
\end{enumerate}  
\end{definition}

Note que $X$ es contráctil si y solamente si $X$ es contráctil en sí mismo. Además, para $A\subset X$, si $A$ es contráctil entonces $A$ es contráctil en $X$. La vuelta no es verdad. % (ver Ejercicio~\ref{exercise:contractil-contractil}). 

\begin{example}
    \noindent\begin{enumerate}
        \item El espacio euclidiano $\mathbb{R}^d$ es contráctil. De hecho, la aplicación $H:\mathbb{R}^d\times [0,1]\to \mathbb{R}^d,~(x,t)\mapsto H(x,t)=(1-t)x$, es una homotopía de la identidad $1_{ \mathbb{R}^d}$ en la aplicación constante $\overline{0}: \mathbb{R}^d\to  \mathbb{R}^d,~x\mapsto 0$.
        \item La esfera $S^{n}=\{(x_1,\ldots,x_{n+1},x_{n+2})\in S^{n+1}:~x_{n+2}=0\}$ es contráctil en $S^{n+1}$. De hecho, la aplicación $H:S^n\times [0,1]\to S^{n+1}, ~(x,t)\mapsto H(x,t)=\dfrac{(1-t)x+tp_N}{\| (1-t)x+tp_N\|}$, donde $p_N=(0,\ldots,0,1)\in S^{n+1}$ es el polo norte, es una homotopía de la inclusión $S^n\hookrightarrow S^{n+1}$ en la aplicación constante $\overline{p_N}:S^n\to S^{n+1},~x\mapsto p_N$.
    \end{enumerate}
\end{example}

El \textit{cono}\index{Cono} para un espacio topológico $X$ es dado por el espacio cociente \[\text{C}(X)=\dfrac{X\times [0,1]}{X\times \{0\}}.\] 

\begin{example}
   Tenemos que el cono $\text{C}(X)$ es contráctil. De hecho, podemos considerar la nulo homotopía $H:\text{C}(X)\times [0,1]\to \text{C}(X)$ dada por \[H([x,s],t)=[x,(1-t)s].\]
\end{example}

Note que, si $X$ es contráctil entonces $X$ es conexo por caminos, o sea, para cualesquier $x,y\in X$ existe una camino continuo $\gamma:[0,1]\to X$ tal que $\gamma(0)=x$ y $\gamma(1)=y$. De hecho, sea $H:X\times [0,1]\to X$ una aplicación continua tal que $H_0=1_X$ y $H_1=\overline{x_0}$ (para algún $x_0\in X$). Considere $\gamma:[0,1]\to X$ por \[\gamma(t)=\begin{cases}
    H(x,2t),&\hbox{si $0\leq t\leq 1/2$;}\\
    H(y,2-2t),&\hbox{si $1/2\leq t\leq 1$.}\\
\end{cases}\] Note que, $\gamma$ es el camino deseado. 

Ahora, veamos que toda aplicación cuyo dominio o codominio sea contráctil es nulo homotópica. 

\begin{proposition}\label{proposition:o-contractil-nulo}
 Sea $f:X\to Y$ una aplicación continua. Si $X$ o $Y$ es contráctil entonces $f$ es nulo homotópica.   
\end{proposition}
\begin{proof}
    Supongamos que $X$ es contractil, o sea, la aplicación identidad $1_X:X\to X$ es nulo homotópica. Note que $f=f\circ 1_X$. Por otro lado, la Proposición~\ref{proposition:o-composta} implica que $f\circ 1_X$ es nulo homotópica y así $f$ es nulo homotópica.

    Ahora, supongamos que $Y$ es nulo homotópica, o sea, la aplicación identidad $1_Y:Y\to Y$ es nulo homotópica. Note que, $f=1_Y\circ f$. Nuevamente la Proposición~\ref{proposition:o-composta} implica que $1_Y\circ f$ es nulo homotópica y así $f$ es nulo homotópica. 
\end{proof}

Además, tenemos que toda aplicación continua con codominio contráctil es homotópica a una aplicación constante deseada. 

\begin{proposition}\label{proposition:codomain-contractil}
Sea $f:X\to Y$ una aplicación continua con $Y$ contráctil y $y_0\in Y$. Entonces $f\simeq \overline{y_0}$. 
\end{proposition}
\begin{proof}
Sea $H:Y\times [0,1]\to Y$ una aplicación continua con $H_0=1_Y$ y $H_1=\overline{y_1}$ (para algún $y_1\in Y$). Considere la aplicación $G:X\times [0,1]\to Y$ dada por \[G(x,t)=\begin{cases}
    H(f(x),2t),&\hbox{si $0\leq t\leq 1/2$;}\\
    H(y_0,2-2t),&\hbox{si $1/2\leq t\leq 1$}.\\
\end{cases}\] Note que $G$ es una homotopía de $f$ en $\overline{y_0}$.
\end{proof}

Proposición~\ref{proposition:codomain-contractil} implica la siguiente afirmación.

\begin{corollary}\label{corollary:contractil-codo-maps-homotopic}
   Sean $f,g:X\to Y$ aplicaciones continua con $Y$ contráctil. Entonces $f\simeq g$. 
\end{corollary}

También, la siguiente afirmación muestra que toda aplicación continua con dominio contráctil y contradominio conexo por caminos es homotópica a una aplicación constante deseada. 

\begin{proposition}\label{proposition:domain-contractil}
Sea $f:X\to Y$ una aplicación continua con $X$ contráctil, $Y$ conexo por caminos y $y_0\in Y$. Entonces $f\simeq \overline{y_0}$. 
\end{proposition}
\begin{proof}
Sea $H:X\times [0,1]\to X$ una aplicación continua con $H_0=1_X$ y $H_1=\overline{x_1}$ (para algún $x_1\in X$). Sea $y_1=f(x_1)$ y considere un camino continuo $\gamma:[0,1]\to Y$ tal que $\gamma(0)=y_1$ y $\gamma(1)=y_0$. Considere la aplicación $G:X\times [0,1]\to Y$ dada por \[G(x,t)=\begin{cases}
    f(H(x,2t)),&\hbox{si $0\leq t\leq 1/2$;}\\
    \gamma(2t-1),&\hbox{si $1/2\leq t\leq 1$}.\\
\end{cases}\] Note que $G$ es una homotopía de $f$ en $\overline{y_0}$.
\end{proof}

Proposición~\ref{proposition:domain-contractil} implica la siguiente afirmación.

\begin{corollary}\label{corollary:contractible-domain-path-conn-codomain-homotopy}
   Sean $f,g:X\to Y$ aplicaciones continua con $X$ contráctil e $Y$ conexo por caminos. Entonces $f\simeq g$. 
\end{corollary}

Consideremos el espacio euclidiano infinito-dimensional $\mathbb{R}^\infty$ dado por \[\mathbb{R}^\infty=\{(x_j):~\text{cada $x_j\in\mathbb{R}$ y existe $n\geq 1$ tal que $x_j=0$ para todo $j\geq n$}\}\] y $\mathbb{R}^n=\{(x_j)\in \mathbb{R}^\infty:~x_j=0 \text{ para todo $j>n$}\}$. La topología de $\mathbb{R}^\infty$ es tal que $A\subset \mathbb{R}^\infty$ es cerrado (equivalentemente, abierto) si, y solamente si, $A\cap\mathbb{R}^n$ es cerrado en $\mathbb{R}^n$ (equivalentemente, abierto), para todo $n\geq 1$. Considere la siguiente aplicación $T:\mathbb{R}^\infty\to\mathbb{R}^\infty$ dada por $T(x_1,x_2,\ldots)=(0,x_1,x_2,\ldots)$ para todo $(x_1,x_2,\ldots)\in\mathbb{R}^\infty$.

\medskip La esfera infinito-dimensional $S^\infty$ es el subespacio de $\mathbb{R}^\infty$ dada por \[S^{\infty}=\{(x_j)\in \mathbb{R}^\infty:~x_1^2+x_2^2+\cdots=1\}.\] Note que tal suma $x_1^2+x_2^2+\cdots$ es finita. 

\medskip Ahora, de \cite[pg.111]{fadell1962}, recordemos el concepto de espacio de configuraciones ordenado. Para $X$ espacio topológico y $k\geq 1$, el \textit{$k$-th espacio de configuraciones ordenado}\index{$k$-th espacio de configuraciones ordenado} de $X$ es el subespacio del producto cartesiano $X^k=X\times\cdots\times X$ dado por:
\[\text{Conf}(X,k)=\{(x_1,\ldots,x_k)\in X^k:~\text{$x_i\neq x_j$ para todo $1\leq i\neq j\leq k$}\}.\] Por convención $\text{Conf}(X,1)=X$.

\medskip El siguiente resultado muestra que el espacio de configuraciones $\text{Conf}(\mathbb{R}^\infty,k)$ es contráctil, para cada $k\in \{1,2\}$. 

\begin{proposition}\label{proposition:conf-infty-2-contractil}
 Para $k\in \{1,2\}$, podemos construir explícitamente una nulo homotopía para $\text{Conf}(\mathbb{R}^\infty,k)$. En particular, tenemos que $\text{Conf}(\mathbb{R}^\infty,k)$ es contráctil.
\end{proposition}
\begin{proof}
    El caso $k=1$ sigue de \cite[Example 1B.3, pg. 88]{hatcher2002}. Veamos para $k= 2$. Considere la aplicación $H:\text{Conf}(\mathbb{R}^\infty,2)\times [0,1]\to \text{Conf}(\mathbb{R}^\infty,2)$ dada por
    \begin{align*}
      H_t\left((x_1,x_2,\ldots),(y_1,y_2,\ldots)\right)&=((1-t)x_1,(1-t)x_2+tx_1,\ldots),\\
      & ((1-t)y_1,(1-t)y_2+ty_1,\ldots)).  
    \end{align*}
%     Recuerde que $T(x_j)=(0,x_{j,1},x_{j,2},\ldots)$ y así $T^\ell(x_j)=(0,\ldots,0,x_{j,1},x_{j,2},\ldots)$ con 0 en las primeras $\ell$ coordenadas, donde $x_j=(x_{j,1},x_{j,2},\ldots)$. 
Veamos que $H$ está bien definida. Supongamos que \[((1-t)x_1,(1-t)x_2+tx_1,\ldots)=((1-t)y_1,(1-t)y_2+ty_1,\ldots).\] Como $t\neq 1$, sigue que $x_{1}=y_{1}$ y así $x_{2}=y_{2}$, $x_{3}=y_{3}, \ldots$, lo cual es una contradicción, ya que $\left((x_1,x_2,\ldots),(y_1,y_2,\ldots)\right)\in \text{Conf}(\mathbb{R}^\infty,2)$. Así $H$ es una homotopía entre la identidad $1_{\text{Conf}(\mathbb{R}^\infty,2)}$ y la aplicación $H_1=T\times T$, recuerde que $T(x_1,x_2,\ldots)=(0,x_{1},x_{2},\ldots)$.  

Ahora, considere la aplicación $G:\text{Conf}(\mathbb{R}^\infty,2)\times [0,1]\to \text{Conf}(\mathbb{R}^\infty,2)$ dada por \begin{align*}
  G_t\left((x_1,x_2,\ldots),(y_1,y_2,\ldots)\right)&=((t,(1-t)x_1,(1-t)x_2,\ldots),(0,(1-t)y_1+t,\\
  & (1-t)y_2,\ldots)).  
\end{align*}  Veamos que $G$ está bien definida. Supongamos que $(t,(1-t)x_1,(1-t)x_2,\ldots)=(0,(1-t)y_1+t,(1-t)y_2,\ldots)$. Luego, $t=0$ y así $x_1=y_1, x_2=y_2,\ldots$, lo cual es una contradicción, ya que $\left((x_1,x_2,\ldots),(y_1,y_2,\ldots)\right)\in \text{Conf}(\mathbb{R}^\infty,2)$. Así, $G$ es una homotopía de $T\times T$ en la aplicación constante $\overline{\left((1,0,0,\ldots),(0,1,0,\ldots)\right)}$. 

Por la demostración del Ítem 3 de la Proposición~\ref{proposition:relacion-homotopia}, tenemos que $F:\text{Conf}(\mathbb{R}^\infty,2)\times [0,1]\to \text{Conf}(\mathbb{R}^\infty,2)$ dada por \[F_t\left((x_1,x_2,\ldots),(y_1,y_2,\ldots)\right)=\begin{cases}
            H_{2t}\left((x_1,x_2,\ldots),(y_1,y_2,\ldots)\right),& \hbox{si $0\leq t\leq 1/2$;}\\
            G_{2t-1}\left((x_1,x_2,\ldots),(y_1,y_2,\ldots)\right),& \hbox{si $1/2\leq t\leq 1$.}\\
        \end{cases}\]  es una nulo homotopía de $1_{\text{Conf}(\mathbb{R}^\infty,2)}$ en $\overline{\left((1,0,0,\ldots),(0,1,0,\ldots)\right)}$. % y así $\text{Conf}(\mathbb{R}^\infty,2)$ es contráctil.
\end{proof}

Por la Proposición~\ref{proposition:conf-infty-2-contractil} formulamos el siguiente problema.

\begin{problem}
\noindent\begin{enumerate}
    \item[(1)]  Para $k\geq 3$, construya explícitamente una nulo homotopía para $\text{Conf}(\mathbb{R}^\infty,k)$.
    \item[(2)] Para cada $k\geq 1$, construya explícitamente una nulo homotopía para $\text{Conf}(S^\infty,k)$. Para el caso $k=1$, i.e., $\text{Conf}(S^\infty,1)=S^\infty$, ver \cite[Example 1B.3, pg. 88]{hatcher2002} o \cite[Theorem 11.1.3, pg. 332]{aguilar2002}.
\end{enumerate}    
\end{problem}

Sea $\mathcal{A}$ una colección de subconjuntos del espacio $X$.
\begin{enumerate}
    \item Se dice que $\mathcal{A}$ es \textit{localmente finita}\index{Localmente finita} en $X$ si para cada $x\in X$ existe un abierto $U$ de $X$ tal que el conjunto \[\{A\in\mathcal{A}:~U\cap A\neq\varnothing\}\] es finito, ver \cite[pg. 278]{munkres2000}. 
    \item Una colección $\mathcal{B}$ de subconjuntos de $X$ se dice que es un \textit{refinamiento}\index{Refinamiento} de $\mathcal{A}$ (o que \textit{refina}\index{Refina} a $\mathcal{A}$) si para cada elemento $B\in\mathcal{B}$ existe un elemento $A\in\mathcal{A}$ tal que $B\subset A$. Si los elementos de $\mathcal{B}$ son conjuntos abiertos, llamamos a $\mathcal{B}$ un \textit{refinamiento abierto}\index{Refinamiento abierto} de $\mathcal{A}$; si son conjuntos cerrados, llamamos a $\mathcal{B}$ un \textit{refinamiento cerrado}\index{Refinamiento cerrado}, ver \cite[pg. 280]{munkres2000}. 
    \item Se dice que $\mathcal{A}$ tiene \textit{orden}\index{Orden} $m+1$ si algún punto de $X$ pertenece a $m+1$ elementos de $\mathcal{A}$, y no existe ningún punto en $X$ que pertenezca a más de $m+1$ elementos de $\mathcal{A}$, ver \cite[pg. 347]{munkres2000}. 
\end{enumerate} 

\medskip Sea $X$ un espacio topológico.
\begin{enumerate}
    \item Se dice que $X$ es \textit{paracompacto}\index{Paracompacto} si todo cubrimiento abierto $\mathcal{A}$ de $X$ tiene un refinamiento abierto localmente finito $\mathcal{B}$ que recubre $X$, ver \cite[pg. 288]{munkres2000}.
    \item Se dice que $X$ es de \textit{dimensión finita}\index{Dimensión finita} si existe algún entero $m$ tal que para todo cubrimiento abierto $\mathcal{A}$ de $X$ existe un cubrimiento abierto $\mathcal{B}$ de $X$ que refina a $\mathcal{A}$ y que tiene orden $m+1$ como máximo. La \textit{dimensión topológica (o de cobertura)}\index{Dimensión topológica o de cobertura} de $X$, denotada por $\text{dim}(X)$, se define como el menor valor $m$ que satisface lo anterior, ver \cite[pg. 348]{munkres2000}.  
\end{enumerate}

\begin{remark}
  De \cite[Theorem 2.1, pg. 29]{zapata2017} tenemos que el espacio de configuraciones $\text{Conf}(M,k)$ nunca es contráctil para ningún $k\geq 2$, donde $M$ es una variedad topológica conexa de dimensión finita. Además, es planteada la siguiente conjetura.  
\end{remark} 

\begin{conjecture}
     Si $X$ es un espacio paracompacto y conexo por caminos con dimensión de cobertura $1\leq \text{dim}(X)<\infty$. Entonces los espacios de configuraciones $\text{Conf}(X,k)$ nunca son contráctiles, para $k\geq 2$.
\end{conjecture}

Adicionalmente, planteamos el siguiente problema.

\begin{problem}
    Sean $B$ un espacio de Banach de dimensión infinita y $k\geq 1$. ¿El espacio de configuraciones $\text{Conf}(B,k)$ es contráctil? 
    
    El caso $k=1$, i.e., $\text{Conf}(B,1)=B$, es inmediato, ya que $B$ es un espacio vectorial topológico.   
\end{problem}

\subsection{Retracto y retracto por deformación}\label{sec:retract}
En esta sección presentamos el concepto y propiedades básicas de retracto y retracto por deformación. Proposición~\ref{proposition:retract-contractible-is-contractible} dice que todo retracto de un espacio contráctil es contráctil. En particular, tenemos que la esfera unitaria $S^{n-1}$ no es un retracto del disco unitario $D^n$ (Proposición~\ref{proposition:esfera-no-retracto}). Como consecuencia, se establece el famoso teorema del punto fijo de Brouwer (Proposición~\ref{proposition:brouwer-fixed-point-theorem}). Proposición~\ref{proposition:retract-extension} presenta una caracterización de retractos en términos de extensión de aplicaciones. Presentamos también el concepto de equivalencia homotópica (Definición~\ref{definition:homotopy-equivalence}). 

\begin{definition}
  Sea $A\subset X$.
  \begin{enumerate}
      \item Diremos que $A$ es un \textit{retracto}\index{Retracto} de $X$ si existe una aplicación continua $r:X\to A$ tal que $r\circ inlc_A=1_A$, donde $incl_A:A\hookrightarrow X$ es la aplicación inclusión, o sea, $r(a)=a$ para todo $a\in A$.  Tal aplicación $r$ es llamada una \textit{retracción}\index{Retracción} de $X$ en $A$. Además, se dice que $X$ \textit{se retrae} en $A$\index{Se retrae}.   
      \item Diremos que $A$ es un \textit{retracto por deformación}\index{Retracto por deformación} de $X$ si existe una aplicación continua $r:X\to A$ tal que $r\circ inlc_A=1_A$ y $incl_A\circ r\simeq 1_X$. Tal aplicación $r$ es llamada una \textit{retracción por deformación}\index{Retracción por deformación} de $X$ en $A$. Además, se dice que $X$ \textit{se retrae por deformación} en $A$\index{Se retrae por deformación}.  
  \end{enumerate} 
\end{definition}

\begin{example}\label{examples-retracts}
    \noindent\begin{enumerate}
        \item[(1)] Todo espacio topológico se retrae en cada uno de sus puntos.
        \item[(2)] Todo espacio topológico contráctil se retrae por deformación en cada uno de sus puntos. Más generalmente, note que si $X$ es contráctil y $A$ es un retracto de $X$ entonces $A$ es un retracto por deformación de $X$. 
        \item[(3)] $\mathbb{R}^{m+1}-\{0\}$ se retrae por deformación en la esfera $S^m$ mediante la retracción $r:\mathbb{R}^{m+1}-\{0\}\to S^m$, $r(x)=\dfrac{x}{\parallel x\parallel}$.
        \item[(4)] $\mathbb{R}^{m+1}-\{-1,1\}$ se retrae por deformación en el wedge de esferas $S^m\vee S^m$.
    \end{enumerate}
\end{example}

 El \textit{cilindro}\index{Cilindro} para un espacio topológico $X$ es dado por el producto cartesiano \[\text{Cil}(X)=X\times [0,1].\] 
 
\begin{example}\label{x-retract-cilindro}
    $X\times \{0\}$ es un retracto por deformación del cilindro $\text{Cil}(X)$. La retracción por deformación es dada por $r:\text{Cil}(X)\to X\times \{0\}$, $r(x,s)=(x,0)$, para cada $(x,s)\in \text{Cil}(X)$. La homotopía entre $i\circ r$ y la identidad $1_{\text{Cil}(X)}$ es dada por $H:\text{Cil}(X)\times [0,1]\to \text{Cil}(X)$, $H(x,s,t)=\left(x,ts\right)$.
\end{example}

Un resultado conocido es que la esfera unitaria $S^{n-1}$ no es un retracto del disco unitario $D^n=\{x\in\mathbb{R}^n:\mid x\|\leq 1\}$. Más generalmente, tenemos la siguiente afirmación. 

\begin{proposition}\label{proposition:retract-contractible-is-contractible}
 Sea $X$ un espacio contráctil y $A\subset X$. Si $A$ es un retracto de $X$ entonces $A$ es contráctil.   
\end{proposition}
\begin{proof}
    Sean $r:X\to A$ una retracción. Note que $1_A=r\circ incl_A$, donde $incl_A:A\hookrightarrow X$. Como $X$ es contráctil entonces la inclusión es nulo homotópica ($r$ también es nulo homotópica), ver Proposición~\ref{proposition:o-contractil-nulo}. Así la identidad $1_A$ es nulo homotópica (ver Proposición~\ref{proposition:o-composta}) y $A$ es contráctil.
\end{proof}

\begin{corollary}\label{corollary:nocontractil-noretracto}
 Sea $X$ un espacio contráctil y $A\subset X$. Si $A$ no es contráctil entonces $A$ no es un retracto de $X$.     
\end{corollary}

Como $D^n$ es contráctil y $S^{n-1}$ no es contráctil, Corolario~\ref{corollary:nocontractil-noretracto} implica la siguiente afirmación.

\begin{proposition}\label{proposition:esfera-no-retracto}
  La esfera unitaria $S^{n-1}$ no es un retracto del disco unitario $D^n$.   
\end{proposition}

Proposición~\ref{proposition:esfera-no-retracto} implica el famoso teorema del punto fijo de Brouwer. Se dice que una aplicación continua $f:X\to X$ tiene \textit{punto fijo}\index{Punto fijo} cuando existe $x\in X$ tal que $f(x)=x$. Caso contrario, diremos que $f:X\to X$ \textit{no tiene punto fijo}. 

\begin{proposition}[Teorema del punto fijo de Brouwer]\label{proposition:brouwer-fixed-point-theorem}
    Si $f:D^n\to D^n$ es una aplicación continua entonces $f$ tiene punto fijo.   
\end{proposition}
\begin{proof}
    Por contradicción. Supongamos que $f(x)\neq x$ para todo $x\in X$. Sea $r:D^n\to S^{n-1}$ dada por $r(x)=$punto de intersección del rayo $\overrightarrow{f(x)x}$ con la esfera $S^{n-1}$, donde $\overrightarrow{f(x)x}$ es el rayo saliendo de $f(x)$ y pasando por $x$. Note que $r$ es continua y $r(x)=x$ para todo $x\in S^{n-1}$. Así, $r$ es una retracción de $D^n$ en $S^{n-1}$, lo cual es una contradicción con la Proposición~\ref{proposition:esfera-no-retracto}.  
\end{proof}

El siguiente ejemplo muestra una aplicación continua de $\mathbb{R}^n$ en $\mathbb{R}^n$ sin punto fijo, para cada $n\geq 1$. 

\begin{example}\label{rn-sin-punto-fijo}
  Para $n\geq 1$, la translación $f:\mathbb{R}^n\to \mathbb{R}^n,~f(x)=x+a$, con $a\in\mathbb{R}^n, a\neq 0$, no tiene punto fijo.  
\end{example}

El siguiente ejemplo muestra una aplicación continua de $(-1,1)$ en $(-1,1)$ sin punto fijo. 

\begin{example}
    La aplicación continua $f:(-1,1)\to (-1,1),~f(s)=\dfrac{s}{2}+\dfrac{1}{2}$, no tiene punto fijo.  
\end{example}

En general, se puede verificar que para cada $n\geq 1$, existe una aplicación continua $f:B^n\to B^n$ sin punto fijo, donde $B^n=\{x\in\mathbb{R}^n:~\mid x\mid<1\}$ es la bola abierta de $\mathbb{R}^n$.% (ver Ejercicio~\ref{exercise:bola-sin-punto-fijo}). 

\medskip Sea $f:X\to Y$ una aplicación continua. Una \textit{sección}\index{Sección} de $f$ es una inversa a derecha de $f$, o sea, una aplicación $s:Y\to X$ tal que $f\circ s=1_Y$. Más generalmente, dado un subespacio $A\subset Y$, una \textit{sección local}\index{Sección local} de $f$ sobre $A$ es una aplicación $s:A\to X$ tal que $f\circ s=incl_A$.

\medskip Sean $X,Y$ espacios topológicos. El conjunto \[\text{Map}(X,Y)=\{f:X\to Y\mid~ f\text{ es continua}\}\] de las aplicaciones continuas es dotado de la \textit{topología compacto-abierta}\index{Topología compacto-abierta} que tiene  como subbase los conjuntos de la forma $U^K=\{f\in \text{Map}(X,Y):~f(K)\subset U\}$, donde $K\subset X$ es compacto y $U$ es un conjunto abierto en $Y$ (ver \cite[Definition 1.2.1, pg. 2]{aguilar2002}).

\begin{remark}\label{remark:contractible-without-fixed-point}
 Para un espacio topológico $X$.
 \begin{enumerate}
     \item[(1)] Existe $\varphi:X\to \text{Conf}(X,2)$ una sección continua de la proyección $\pi:\text{Conf}(X,2)\to X,~\pi(x,y)=x$, si, y solamente si, existe una aplicación continua $f:X\to X$ sin punto fijo.
     \item[(2)] De \cite[Proposition 4.2, pg. 571]{zapata2020} se tiene que, si existe $s:\text{Conf}(X,2)\times X\to\text{Map}([0,1],\text{Conf}(X,2))$ una sección continua de la aplicación $e_{\pi}:\text{Map}([0,1],\text{Conf}(X,2))\to \text{Conf}(X,2)\times X,~e_\pi(\gamma)=\left(\gamma(0),\pi(\gamma(1))\right)$, entonces $X$ es contráctil y existe una aplicación continua $f:X\to X$ sin punto fijo. Además, si la proyección $\pi:\text{Conf}(X,2)\to X,~\pi(x,y)=x$ es una fibración, se tiene que el regreso también vale (ver \cite[Corollary 4.8, pg. 574]{zapata2020}).  
 \end{enumerate}  
\end{remark}

Debido al Ítem (2) de la Observación~\ref{remark:contractible-without-fixed-point} planteamos el siguiente problema.

\begin{problem}
  Sea $X$ un espacio contráctil tal que existe una aplicación continua $f:X\to X$ sin punto fijo. ¿Existe $s:\text{Conf}(X,2)\times X\to\text{Map}([0,1],\text{Conf}(X,2))$ sección continua de la aplicación $e_{\pi}:\text{Map}([0,1],\text{Conf}(X,2))\to \text{Conf}(X,2)\times X,~e_\pi(\gamma)=\left(\gamma(0),\pi(\gamma(1))\right)$?  
\end{problem}

\medskip Sean $X$, $Y$ espacios topológicos, $A\subset X$ y $g:A\to Y$ una aplicación continua. Una \textit{extensión}\index{Extensión} de $g$ sobre $X$ es una aplicación continua $f:X\to Y$ tal que $f(a)=g(a)$, para todo $a\in A$. De \cite[Proposition 7.1, pg. 10]{hu1959} tenemos el siguiente resultado. 

\begin{proposition}\label{proposition:retract-extension}
 Sea $X$ un espacio topológico y $A\subset X$. $A$ es un retracto de $X$ si, y solamente si, para cualquier espacio topológico $Y$, cada aplicación continua $g:A\to Y$  tiene una extensión sobre $X$.  
\end{proposition}
\begin{proof}
  Veamos la ida. Sea $r:X\to A$ una retracción. Tenemos que $g\circ r:X\to Y$ es una extensión de $g$ sobre $X$, ya que, para $a\in A$, tenemos $(g\circ r)(a)=g(r(a))=g(a)$. Ahora, veamos el regreso. Para $g=1_A$, por hipótesis tenemos que, existe $r:X\to A$ extensión de $1_A$ sobre $X$. Note que, $r$ es una retracción. 
\end{proof}

\begin{definition}[Equivalencia homotópica]\label{definition:homotopy-equivalence}
Sean $X$, $Y$ espacios topológicos. \begin{enumerate}
    \item Se dice que $X$ \textit{domina}\index{Domina} a $Y$ si existen aplicaciones continuas $f:X\to Y$ y $g:Y\to X$ tales que $f\circ g\simeq 1_Y$.
    \item Una aplicación continua $f:X\to Y$ es llamada una \textit{equivalencia homotópica}\index{Equivalencia homotópica} si existe una aplicación continua $g:Y\to X$ tal que $g\circ f\simeq 1_X$ y $f\circ g\simeq 1_Y$. Tal $g$ es llamada \textit{inversa homotópica}\index{Inversa homotópica} de $f$.
    \item Se dice que $X$ es \textit{equivalente homotópico}\index{Equivalente homotópico} a $Y$, denotado por $X\simeq Y$, cuando existe una equivalencia homotópica $f:X\to Y$.  
\end{enumerate}
\end{definition}

Note que la relación \aspas{$\simeq$} es una relación de equivalencia. Así, cuando $X\simeq Y$ diremos que $X$ e $Y$ tienen el \textit{mismo tipo de homotopía}\index{Mismo tipo de homotopía}. 

\begin{example}
    \noindent\begin{enumerate}
        \item Todo espacio topológico domina a sus retractos.
        \item Si $A$ es un retracto por deformación de $X$ entonces $A$ y $X$ tienen el mismo tipo de homotopía, i.e., $A\simeq X$. 
    \end{enumerate}
\end{example}

\begin{example}
   Sea $X$ un espacio topológico. \begin{enumerate}
       \item  $X$ es contráctil si, y solamente si, $X\simeq \{\ast\}$.
       \item $\text{C}(X)\simeq \{\ast\}$.
       \item $\text{Cil}(X)\simeq X$.
   \end{enumerate} 
\end{example}

%%%%%%%%%%%%%%%%%%%%%%%%%%%%%%%%%%%%%%%%%%%%%%%%%%%%%%%%%%%%%%%%%
%%%%%%%%%%%%%%%%%%%%%%%%%%%%%%%%%%%%%%%%%%%%%%%%%%%%%%%%%%%%%%%%%%
\section{Categoría de una aplicación}\label{chap:categoria-aplicacion}
En este capítulo daremos una revisión de categoría de una aplicación. En la Sección~\ref{sec:definiciones-cat} presentamos la definición y propiedades básicas de la categoría de una aplicación. En la Sección~\ref{sec:categoriaLS} recordamos la definición y propiedades básicas de categoría LS. 

\subsection{Definiciones básicas}\label{sec:definiciones-cat}
En esta sección presentamos la definición y propiedades básicas de la categoría de una aplicación. Una propiedad es que la categoría de una aplicación es un invariante homotópico (Proposición~\ref{proposition:cat-aplicacao-invariante-homotopico}). La Proposición~\ref{proposition:cat-composta} muestra el comportamiento de la categoría con respecto a la composición. Además, Proposición~\ref{proposition:cat-composta-equivalencia-homotopica} dice que si componemos a una aplicación con una equivalencia homotópica su categoría no se altera. Se cumple que la categoría de una aplicación se puede realizar como la categoría de una fibración (ver Proposición~\ref{proposition:cat-map-cat-fibration}).

De \cite[Definition 1.1, pg. 265]{berstein1962} tenemos la siguiente definición.

\begin{definition}
    Sea $f:X\to Y$ una aplicación continua. La \textit{categoría}\index{Categoría} de $f$, denotada por $\mathrm{cat}(f)$, es el menor entero positivo $n$ tal que existen $U_1,\ldots,U_n$ abiertos de $X$ tales que $X=U_1\cup\cdots\cup U_n$ y cada aplicación restricción $f_{| U_i}:U_i\to Y$ es nulo homotópica. En caso que no exista tal entero $n$, escribiremos $\mathrm{cat}(f)=\infty$.
\end{definition}

Note que $\mathrm{cat}(f)=1$ si y solamente si $f$ es nulo homotópica.

\begin{example}
\noindent\begin{enumerate}
    \item[(1)] La categoría de cualquier aplicación constante $\overline{y_0}:X\to Y,~\overline{y_0}(x)=y_0$, es igual a $1$.
    \item[(2)]  Sea $f:X\to Y$ una aplicación continua. Si $X$ o $Y$ es contráctil, entonces, por la Proposición~\ref{proposition:o-contractil-nulo}, tenemos que $\text{cat}(f)=1$. 
\end{enumerate} 
\end{example}

Tenemos que la categoría $\mathrm{cat}(-)$ es un invariante homotópico.

\begin{proposition}\label{proposition:cat-aplicacao-invariante-homotopico}
 Sean $f,g:X\to Y$ aplicaciones continuas. Si $f\simeq g$ entonces $\mathrm{cat}(f)=\mathrm{cat}(g)$.   
\end{proposition}

\begin{proof}
 Veamos que $\mathrm{cat}(f)\leq\mathrm{cat}(g)$. Supongamos que $\mathrm{cat}(g)=n$ y consideremos $U_1,\ldots,U_n$ abiertos de $X$ tales que $X=U_1\cup\cdots\cup U_n$ y cada aplicación restricción $g_{| U_i}:U_i\to Y$ es nulo homotópica. Note que $g_{| U_i}=g\circ incl_{U_i}$ para cada $i=1,\ldots,n$, donde $incl_{U_i}:U_i\hookrightarrow X$ es la aplicación inclusión. Luego, \begin{align*}
     f_{| U_i}&= f\circ incl_{U_i}\\
     &\simeq g\circ incl_{U_i} & \mbox{ (sigue del Corolario~\ref{corollary:homo-igual})}\\
     &=g_{| U_i},
 \end{align*} así $f_{| U_i}$ es nulo homotópica, ya que $g_{| U_i}$ es nulo homotópica. Luego, $\mathrm{cat}(f)\leq n=\mathrm{cat}(g)$. Análogamente, se tiene que $\mathrm{cat}(g)\leq\mathrm{cat}(f)$. Por lo tanto, $\mathrm{cat}(f)=\mathrm{cat}(g)$.
\end{proof}

El regreso de la Proposición~\ref{proposition:cat-aplicacao-invariante-homotopico} no vale. Así, es natural plantear el siguiente problema.

\begin{problem}
  Sean $f,g:X\to Y$ aplicaciones continuas. ¿Existen condiciones para $f,g$ tales que el regreso de la Proposición~\ref{proposition:cat-aplicacao-invariante-homotopico} sea verdad, o sea, si $\mathrm{cat}(f)=\mathrm{cat}(g)$ implica que $f\simeq g$?  
\end{problem}

Ahora veamos el comportamiento de $\mathrm{cat}(-)$ con respecto a la composición. 

\begin{proposition}\label{proposition:cat-composta}
 Sean $f:X\to Y$ y $g:Y\to Z$ aplicaciones continuas. Tenemos \[\mathrm{cat}(g\circ f)\leq\min\{\mathrm{cat}(f),\mathrm{cat}(g)\}.\]   
\end{proposition}
\begin{proof}
  Veamos que $\mathrm{cat}(g\circ f)\leq \mathrm{cat}(f)$. Supongamos que $\mathrm{cat}(f)=n$ y consideremos $U_1,\ldots,U_n$ abiertos de $X$ tales que $X=U_1\cup\cdots\cup U_n$ y cada aplicación restricción $f_{| U_i}:U_i\to Y$ es nulo homotópica. Note que \begin{align*}
      (g\circ f)_{| U_i}&=(g\circ f)\circ incl_{U_i}\\
      &=g\circ (f\circ incl_{U_i})\\
      &=g\circ f_{| U_i},
  \end{align*} luego, aplicando la Proposición~\ref{proposition:o-composta}, obtenemos que $ (g\circ f)_{| U_i}$ es nulo homotópica, ya que $f_{| U_i}$ es nulo homotópica. Así, $\mathrm{cat}(g\circ f)\leq n=\mathrm{cat}(f)$.

  Ahora, veamos que $\mathrm{cat}(g\circ f)\leq \mathrm{cat}(g)$. Supongamos que $\mathrm{cat}(g)=n$ y consideremos $V_1,\ldots,V_n$ abiertos de $Y$ tales que $Y=V_1\cup\cdots\cup V_n$ y cada aplicación restricción $g_{| V_i}:V_i\to Z$ es nulo homotópica. Para cada $i=1,\ldots,n$, considere $U_i=f^{-1}(V_i)$. Note que, cada $V_i$ es abierto de $X$ y $X=U_1\cup\cdots\cup U_n$. Además, \begin{align*}
     (g\circ f)_{| U_i}&=(g\circ f)\circ incl_{U_i}\\
     &=g\circ (f\circ incl_{U_i})\\
     &=g\circ\left(incl_{V_i}\circ f_{|}\right)\\
     &=\left(g\circ incl_{V_i}\right)\circ f_{|}\\
     &=g_{| V_i}\circ f_{|},
  \end{align*} donde $f_{|}:U_i\to V_i$ está dada por $f_{|}(x)=f(x)$ para todo $x\in U_i$. Luego, nuevamente aplicando la Proposición~\ref{proposition:o-composta}, obtenemos que $ (g\circ f)_{| U_i}$ es nulo homotópica, ya que $g_{| V_i}$ es nulo homotópica. Así, $\mathrm{cat}(g\circ f)\leq n=\mathrm{cat}(g)$.

  Por lo tanto, $\mathrm{cat}(g\circ f)\leq\min\{\mathrm{cat}(f),\mathrm{cat}(g)\}$.   
\end{proof}

Por la demostración de la Proposición~\ref{proposition:cat-composta} tenemos que la Proposición~\ref{proposition:o-composta} implica la Proposición~\ref{proposition:cat-composta}. La vuelta también vale. % (ver Ejercicio~\ref{exercise:vuelta-cat-composta-implica-nulo}).

\medskip El siguiente ejemplo muestra que la desigualdad en la Proposición~\ref{proposition:cat-composta} puede ser estricta.

\begin{example}
    Considere el wedge $S^n\vee S^n=S^n\times\{1\}\cup\{1\}\times S^n$ junto con las aplicaciones $g:S^n\vee S^n\to S^n\vee S^n$ y $f:S^n\vee S^n\to S^n\vee S^n$ dadas por \begin{align*}
        f(x,y)=\begin{cases}
        (1,1),&\mbox{si $y=1$;}\\
        (1,y),&\mbox{si $x=1$.}\\
    \end{cases}\\
    g(x,y)=\begin{cases}
        (x,1),&\mbox{si $y=1$;}\\
        (1,1),&\mbox{si $x=1$.}\\
    \end{cases}
    \end{align*} Note que $f$ y $g$ no son nulo homotópicas (ver Ejercicio~\ref{exercise:no-nulo-composta-nulo}). Por otro lado $f\circ g=\overline{(1,1)}$. Así, $\mathrm{cat}(f\circ g)=1$ y $\min\{\mathrm{cat}(f),\mathrm{cat}(g)\}\geq 2$.
\end{example}

\medskip Además, la Proposición~\ref{proposition:cat-composta} implica una secuencia de números naturales limitada y decreciente $\left(\mathrm{cat}(f^{n})\right)_{n\geq 1}$, donde $f^n$ es la $n$-th iterada, $f^n=f\circ\cdots\circ f$, de una autoaplicación $f:X\to X$.

\begin{corollary}\label{corollary:cat-iterada}
    Sea $f: X\to X$ una aplicación continua. Para cualquier $n\geq 1$, tenemos \[1\leq\cdots\leq\mathrm{cat}(f^{n+1})\leq\mathrm{cat}(f^{n})\leq\cdots\leq\mathrm{cat}(f^2)\leq\mathrm{cat}(f).\]
\end{corollary}

No conocemos un ejemplo de una autoaplicación $f:X\to X$ tal que $\mathrm{cat}(f\circ f)<\mathrm{cat}(f)$. Del Corolario~\ref{corollary:cat-iterada} tenemos que existe $n_0\geq 1$ tal que $\mathrm{cat}(f^{n})=\mathrm{cat}(f^{n_0})$, para todo $n\geq n_0$. Luego, podemos dar la siguiente definición.

\begin{definition}\label{definition:cat-f-infty}
   Sea $f: X\to X$ una aplicación continua. Denotemos por \begin{align*}
       \mathrm{cat}_f&=\min\{n_0:~\mathrm{cat}(f^{n})=\mathrm{cat}(f^{n_0})\text{, para todo $n\geq n_0$}\}.\\
       \mathrm{cat}(f^{\infty})&=\mathrm{cat}(f^{\mathrm{cat}_f}).
   \end{align*} \[\] 
\end{definition}

 $\mathrm{cat}_f$ y $\mathrm{cat}(f^{\infty})$ son nuevos invariantes numéricos que no han sido estudiados. Así, presentamos el siguiente problema.

\begin{problem}
  Sea $f: X\to X$ una aplicación continua. Estudiar $\mathrm{cat}_f$ y $\mathrm{cat}(f^{\infty})$.  
\end{problem}

\medskip Ahora veamos que, si componemos a una aplicación con una equivalencia homotópica su categoría no se altera.

\begin{proposition}\label{proposition:cat-composta-equivalencia-homotopica}
    Sean  $f:X\to Y$ y $g:Y\to Z$ aplicaciones continuas. \begin{enumerate}
        \item[(1)] Si existe $h:Y\to X$ continua tal que $f\circ h\simeq 1_Y$ entonces $\mathrm{cat}(g\circ f)=\mathrm{cat}(g)$. 
        \item[(2)] Si existe $g':Z\to Y$ continua tal que $g'\circ g\simeq 1_Y$ entonces $\mathrm{cat}(g\circ f)=\mathrm{cat}(f)$. 
    \end{enumerate}
\end{proposition}
\begin{proof}
    Supongamos que $\mathrm{cat}(g\circ f)=k$ y considere $W_1,\ldots,W_k$ abiertos de $X$ tales que $X=W_1\cup\cdots\cup W_k$ y cada restricción $(g\circ f)_{|W_j}:W_j\to Z$ es nulo homotópica. 
    \begin{enumerate}
        \item[(1)] Para cada $j=1,\ldots,k$, consideremos $U_j=h^{-1}(W_j)$. Note que cada $U_j$ es abierto de $Y$ y $Y=U_1\cup\cdots\cup U_k$. Además, para cada $j=1,\ldots,k$, tenemos:\begin{align*}
            g_{| U_j}&=g\circ incl_{U_j}\\
            &=(g\circ 1_Y)\circ incl_{U_j}\\
            &\simeq \left(g\circ (f\circ h)\right)\circ incl_{U_j} & \mbox{ (sigue del Corolario~\ref{corollary:homo-igual}, ya que  $f\circ h\simeq 1_Y$)}\\
            &=\left((g\circ f)\circ h\right)\circ incl_{U_j}\\
            &=(g\circ f)\circ (h\circ incl_{U_j})\\
            &=(g\circ f)\circ (incl_{W_j}\circ h_{|})\\
             &=\left((g\circ f)\circ incl_{W_j}\right)\circ h_{|}\\
             &=(g\circ f)_{| W_j}\circ h_{|},
        \end{align*} donde $h_{|}:U_j\to W_j$ está dada por $h_{|}(x)=h(x)$, para todo $x\in U_j$. Así, por la Proposición~\ref{proposition:o-composta}, sigue que $g_{| U_j}$ es nulo homotópica, ya que $(g\circ f)_{|W_j}$ es nulo homotópica. Luego, $\mathrm{cat}(g)\leq k=\mathrm{cat}(g\circ f)$. Además, por la Proposición~\ref{proposition:cat-composta}, obtenemos que $\mathrm{cat}(g\circ f)=\mathrm{cat}(g)$.
        \item[(2)] Para cada $j=1,\ldots,k$, tenemos:
        \begin{align*}
            f_{| W_j}&=f\circ incl_{W_j}\\
            &=(1_Y\circ f)\circ incl_{W_j}\\
             &\simeq\left((g'\circ g)\circ f\right)\circ incl_{W_j}& \mbox{ (sigue del Corolario~\ref{corollary:homo-igual}, ya que  $g'\circ g\simeq 1_Y$)}\\
             &=\left(g'\circ (g\circ f)\right)\circ incl_{W_j}\\
             &=g'\circ\left((g\circ f)\circ incl_{W_j}\right)\\
             &=g'\circ(g\circ f)_{| W_j}.
        \end{align*} Así, nuevamente por la Proposición~\ref{proposition:o-composta}, sigue que $f_{| W_j}$ es nulo homotópica, ya que $(g\circ f)_{|W_j}$ es nulo homotópica. Luego, $\mathrm{cat}(f)\leq k=\mathrm{cat}(g\circ f)$. Además, nuevamente por la Proposición~\ref{proposition:cat-composta}, obtenemos que $\mathrm{cat}(g\circ f)=\mathrm{cat}(f)$.
    \end{enumerate}
\end{proof}

Proposición~\ref{proposition:cat-composta-equivalencia-homotopica} implica directamente el siguiente resultado.

\begin{corollary}\label{corollary:cat-funcion-composta-retraccion}
    Sea  $f:X\to Y$ una aplicación continua. \begin{enumerate}
        \item[(1)] Si $r:Z\to X$ es una retracción de $Z$ en $X$ entonces $\mathrm{cat}(f\circ r)=\mathrm{cat}(f)$. 
        \item[(2)] Si $Y$ es un retracto de $W$ entonces $\mathrm{cat}(incl_Y\circ f)=\mathrm{cat}(f)$, donde $incl_Y:Y\hookrightarrow W$ es la inclusión. 
    \end{enumerate}
\end{corollary}

\begin{example}
Sea $f:X\to Y$ una aplicación continua.
\begin{enumerate}
    \item[(1)]  Del Ejemplo~\ref{x-retract-cilindro} tenemos que $X$ es un retracto por deformación del cilindro $\text{Cil}(X)$ cuya retracción por deformación está dada por $r:\text{Cil}(X)\to X$, $r(x,s)=x$, para cada $(x,s)\in \text{Cil}(X)$. Sea $F:\text{Cil}(X)\to Y,~F(x,s)=f(x)$ y note que $F=f\circ r$. Así, por el Corolario~\ref{corollary:cat-funcion-composta-retraccion} obtenemos que $\mathrm{cat}(F)=\mathrm{cat}(f)$. Note también que $f=F\circ incl_X$, donde $incl_X:X\hookrightarrow \text{Cil}(X),~incl_X(x)=(x,0)$, es la inclusión. 
\[
\xymatrix{ & \mathrm{Cil}(X) \ar[rd]^{F} \ar@/_10pt/[d]_{r} & & \\ 
&  X \ar@{^{(}->}[u]_{ } \ar[r]_{f} &  Y & }
\]
\item[(2)] Nuevamente, por el Ejemplo~\ref{x-retract-cilindro} tenemos que $Y$ es un retracto por deformación del cilindro $\text{Cil}(Y)$. Así, por la Proposición~\ref{proposition:cat-composta-equivalencia-homotopica} obtenemos que $\mathrm{cat}(incl_Y\circ f)=\mathrm{cat}(f)$, ya que la inclusión $incl_Y:Y\hookrightarrow \text{Cil}(Y),~incl_Y(y)=(y,0)$, es una equivalencia homotópica.  
\[
\xymatrix{ &    &\mathrm{Cil}(Y) & \\ 
&  X \ar[ru]^{incl_Y\circ f} \ar[r]_{f} &  Y\ar@{^{(}->}[u]_{incl_Y} & }
\]
\end{enumerate} 
\end{example} %@{^{(}->} @C=3cm \rule{3mm}{0mm} \ar@/_10pt/[d]_{r}

De \cite[p. 407]{hatcher2002} recordemos la siguiente construcción y definición.

\begin{definition}
    Para una aplicación continua $f:X\to Y$, considere el espacio  
\begin{equation*}
E_f=\{(x,\gamma)\in X\times \mathrm{Map}([0,1],Y)\mid~\gamma(0)=f(x)\},  
\end{equation*} donde $\mathrm{Map}([0,1],Y)$ es el espacio de todos los caminos $[0,1]\to Y$. La aplicación \begin{equation*}
\rho_f:E_f\to Y,~(x,\gamma)\mapsto \rho_f(x,\gamma)=\gamma(1)
\end{equation*} es una fibración.
\end{definition}

 Además, la proyección sobre la primera coordenada $\pi:E_f\to X,~(x,\gamma)\mapsto x$ es una equivalencia homotópica con inversa homotópica $c:X\to E_f$ dada por $x\mapsto (x,\gamma_{f(x)})$, donde $\gamma_{f(x)}$ es el camino constante en $f(x)$. Así, tenemos una factorización $$\left(X\stackrel{f}\to Y\right)\,=\,\left(X\stackrel{c}{\to} E_f\stackrel{\rho_f}{\to} Y\right),$$ la cual es una composición de una equivalencia homotópica seguida por una fibración. Además, note que la proyección $\pi:E_f\to X,~(x,\gamma)\mapsto x$ es una retracción por deformación de $E_f$ en $X$. Luego, aplicando la Proposición~\ref{proposition:cat-composta-equivalencia-homotopica} obtenemos el siguiente resultado.

\begin{proposition}\label{proposition:cat-map-cat-fibration}
    Sea $f:X\to Y$ una aplicación continua. Tenemos que \[\mathrm{cat}(f)=\mathrm{cat}(\rho_f).\]
\end{proposition}

De \cite[p. 348]{hatcher2002}, una \textit{aplicación celular}\index{Aplicación celular} $f:X\to Y$, entre complejos CW, es una aplicación continua tal que $f(X^q)\subset Y^q$, para todo $q\geq 0$. Aquí $X^q$ denota el $q$-esqueleto de $X$, o sea, el conjunto de todas la células de dimensión menor o igual que $q$. 

\begin{remark}\label{remark:teorema-aproximacion-celular}
 Sean $X$, $Y$ complejos CW y $f:X\to Y$ una aplicación continua. De  \cite[Theorem 4.8, p. 349]{hatcher2002}, Teorema de aproximación celular, tenemos que existe una aplicación celular $g:X\to Y$ tal que $g\simeq f$. 
\end{remark}

\begin{example}\label{cat-map-esferas}
    Sea $f:S^n\to S^m$ una aplicación continua, con $n<m$. Veamos que $\mathrm{cat}(f)=1$. Consideremos sobre $S^n$ su estructura celular estándar, o sea, una única célula de dimensión $0$ y una única célula de dimensión $n$. Así, $(S^n)^q=\{\ast\}$, para cada $q<n$, y $(S^n)^n=S^n$. Por el Teorema de aproximación celular (ver Observación~\ref{remark:teorema-aproximacion-celular}), existe una aplicación celular $g:S^n\to S^m$ tal que $g\simeq f$. Note que $g(S^n)\subset\{\ast\}$. Así, $g$ es una aplicación constante. Luego, $f$ es nulo homotópica.  
\end{example}

\subsection{Categoría de Lusternik-Schnirelmann}\label{sec:categoriaLS}
En esta sección recordamos la definición (Definición~\ref{definition:LS-category-def}) y propiedades básicas de categoría LS. La categoría LS da una cota superior para la categoría de una aplicación (Proposición~\ref{proposition:cat-catLS}). En particular, se cumple que la categoría de una aplicación con dominio o codominio una esfera es a lo máximo 2 (Corolario~\ref{corollary:cat-do-codo-esfera}). La Proposición~\ref{proposition:cat-no-nulo-2} da una caracteización de cuando una aplicación continua, con dominio o codominio una esfera, no es nulo homotópica. El comportamiento de la categoría LS para los retractos está dada en la Proposición~\ref{proposition:x-domina-y-catLS}. En particular se muestra que la categoría LS es un invariante homotópico (Corolario~\ref{corollary:cat-space-invariante}). Además, se tiene que la categoría LS de un espacio siempre le gana a la categoría LS de sus retractos (Corolario~\ref{corollary:catLS-retractos}). La Proposición~\ref{proposition:cat-map-inversa-homotopica} muestra la categoría de una equivalencia homotópica. En la Definición~\ref{definition:secat-def} recordamos la noción de categoría seccional. Una conexión entre categoría, categoría seccional y categoría LS está dada en el Teorema~\ref{theorem:cat-secat-maior-cat}. La Proposición~\ref{proposition:secat-1-cat-igual-catLS} muestra una igualdad entre la categoría de una aplicación, con categoría seccional uno, y la categoría LS de su codomoinio. En particular, se tiene que la categoría de la proyección sobre la primera coordenada definida sobre el espacio de configuraciones ordenado con dos puntos coincide con la categoría del espacio, siempre que tal espacio admita una aplicación continua sin punto fijo (Proposición~\ref{proposition:cat-proyection-section}).  
 
\medskip De \cite[Definition 1.1, pg. 1]{cornea2003} recordemos la noción de categoría de Lusternik-Schnirelmann o simplemente categoría LS de un espacio topológico. 

\begin{definition}[Categoría LS]\label{definition:LS-category-def}
   Sea $X$ un espacio topológico. La \textit{categoría LS}\index{Categoría LS} de $X$, denotado por $\mathrm{cat}(X)$, es el menor entero positivo $k$ tal que existen $U_1,\ldots,U_k$ abiertos de $X$ tales que $X=U_1\cup\cdots\cup U_k$ y cada inclusión $U_j\hookrightarrow X$ es nulo homotópica. En el caso que no exista tal entero $k$ escribiremos $\text{cat}(X)=\infty$. 
\end{definition}

Note que $\mathrm{cat}(1_X)=\mathrm{cat}(X)$ para cualquier espacio topológico $X$. Además, $\mathrm{cat}(X)=1$ si y solamente si $X$ es contráctil.

\begin{example}\label{exam:catLS}
\noindent\begin{enumerate}
    \item[(1)] $\mathrm{cat}(\mathbb{R}^n)=1$ para cualquier $n\geq 0$.
    \item[(2)] $\mathrm{cat}(S^m)=2$ para todo $m\geq 0$. De hecho, como $S^m$ no es contráctil, tenemos que $\mathrm{cat}(S^m)\geq 2$. Por otro lado, considere $U_1=S^{m}\setminus\{p_N\}$ y $U_2=S^{m}\setminus\{p_S\}$, donde $p_N$ y $p_S$ son el polo norte y sur, respectivamente. Claramente, cada $U_i$ es abierto de $S^m$ y $S^m=U_1\cup U_2$. Además, cada $U_i$ es contráctil, ya que es homeomorfo a $\mathbb{R}^m$, usando la proyección estereográfica. En particular, cada $U_i$ es contráctil en $S^m$. Así, $\mathrm{cat}(S^m)\leq 2$. 
\end{enumerate}    
\end{example}

Proposición~\ref{proposition:cat-composta} implica la siguiente afirmación.

\begin{proposition}\label{proposition:cat-catLS}
    Sea $f:X\to Y$ una aplicación continua. Tenemos que \[\text{cat}(f)\leq\min\{\mathrm{cat}(X),\mathrm{cat}(Y)\}.\]
\end{proposition}
\begin{proof}
    Note que $f=f\circ 1_X$ y $f=1_Y\circ f$. Luego, por la Proposición~\ref{proposition:cat-composta}, tenemos que $\text{cat}(f)\leq\text{cat}(X)$ y $\text{cat}(f)\leq\text{cat}(Y)$.
\end{proof}

\begin{example}
    Sea $X$ un espacio topológico y $\Delta:X\to X\times X,~\Delta(x)=(x,x)$, la aplicación diagonal. Tenemos que $\mathrm{cat}(\Delta)=\mathrm{cat}(X)$. De hecho, por Proposición~\ref{proposition:cat-catLS}, tenemos que $\mathrm{cat}(\Delta)\leq\mathrm{cat}(X)$. Por otro lado, note que $1_X=p_1\circ\Delta$, donde $p_1:X\times X\to X,~p_1(x_1,x_2)=x_1$, es la proyección en la primera coordenada. Luego, por la Proposición~\ref{proposition:cat-composta}, sigue que $\mathrm{cat}(X)=\mathrm{cat}(1_X)\leq\mathrm{cat}(\Delta)$. 
\end{example}

Ahora veamos que la categoría de una aplicación con dominio o codominio una esfera es a lo máximo 2.

\begin{corollary}\label{corollary:cat-do-codo-esfera}
 Sea $f:X\to Y$ una aplicación continua. Si $X$ o $Y$ es una esfera entonces \[\text{cat}(f)\in\{1,2\}.\]   
\end{corollary}
\begin{proof}
    Del Ítem (2) del Ejemplo\ref{exam:catLS} tenemos que $\text{cat}(X)=2$ o $\text{cat}(Y)=2$. Luego, por la Proposición~\ref{proposition:cat-catLS}, sigue que $\text{cat}(f)\leq 2$.
\end{proof}

Así, podemos dar una caracterización de cuando una aplicación continua, con dominio o contradominio una esfera, no es nulo homotópica.

\begin{proposition}\label{proposition:cat-no-nulo-2}
  Sea $f:X\to Y$ una aplicación continua con $X$ o $Y$ una esfera. Tenemos que $\text{cat}(f)=2$ si, y solamente si, $f$ no es nulo homotópica.   
\end{proposition}
\begin{proof}
    Por la definición de categoría, note que, si $\text{cat}(f)=2$ entonces $f$ no es nulo homotópica. Ahora, si $f$ no es nulo homotópica entonces, nuevamente por la definición de categoría, sigue que $\text{cat}(f)\geq 2$. Luego, por el Corolario~\ref{corollary:cat-do-codo-esfera}, concluimos que $\text{cat}(f)=2$. 
\end{proof}

Ahora, veamos el comportamiento de la categoría LS para los retractos. Para ello, tenemos el siguiente resultado.

\begin{proposition}\label{proposition:x-domina-y-catLS}
  Sean $X, Y$ espacios topológicos. Si $X$ domina a $Y$, entonces \[\text{cat}(Y)\leq\text{cat}(X).\] Equivalentemente, si  $\text{cat}(Y)>\text{cat}(X)$ entonces $X$ no domina a $Y$.
\end{proposition}
\begin{proof}
    Sean $Y\stackrel{g}{\to} X\stackrel{f}{\to} Y$ aplicaciones continuas tales que $f\circ g\simeq 1_Y$. Tenemos \begin{align*}
        \text{cat}(Y)&=\text{cat}(1_Y)\\
        &=\text{cat}(f\circ g) & \mbox{ (sigue de la Proposición~\ref{proposition:cat-aplicacao-invariante-homotopico})}\\
        &\leq\min\{\text{cat}(f),\text{cat}(g)\} & \mbox{ (sigue de la Proposición~\ref{proposition:cat-composta})}\\
        &\leq\text{cat}(X) & \mbox{ (sigue de la Proposición~\ref{proposition:cat-catLS})}
    \end{align*}
\end{proof}

En particular, obtenemos que la categoria LS es un invariante homotópico.

\begin{corollary}\label{corollary:cat-space-invariante}
Sean $X,Y$ espacios topológicos. Si $X\simeq Y$ entonces $\mathrm{cat}(X)=\mathrm{cat}(Y)$.     
\end{corollary} 

La vuelta del Corolario~\ref{corollary:cat-space-invariante} no es válida. Luego, es natural formular el siguiente problema.

\begin{problem}
   Sean $X,Y$ espacios topológicos. Existen condiciones para $X$ e $Y$ tales que la vuelta del Corolario~\ref{corollary:cat-space-invariante} sea válida, o sea, si $\mathrm{cat}(X)=\mathrm{cat}(Y)$ implique que $X\simeq Y$? 
\end{problem}

Además, tenemos que la categoría LS de un espacio siempre le gana a la categoría LS de sus retractos.

\begin{corollary}\label{corollary:catLS-retractos}
Si $A$ es un retracto de $X$ entonces $\mathrm{cat}(A)\leq\mathrm{cat}(X)$.     
\end{corollary} 

\begin{example}
    Como $\text{cat}(D^n)=1$ y $\text{cat}(S^{n-1})=2$, por el Corolario~\ref{corollary:catLS-retractos}, podemos concluir que la esfera $S^{n-1}$ no es un retracto del disco $D^n$.  
\end{example}

Como una aplicación directa de la Proposición~\ref{proposition:cat-aplicacao-invariante-homotopico} junto con la Proposición~\ref{proposition:cat-composta-equivalencia-homotopica}, tenemos la siguiente afirmación.  

\begin{proposition}\label{proposition:cat-map-inversa-homotopica}
    Sea $f:X\to Y$ una equivalencia homotópica con inversa homotópica $g:Y\to X$. Entonces, $\mathrm{cat}(f)=\mathrm{cat}(g)=\mathrm{cat}(X)=\mathrm{cat}(Y)$.
\end{proposition}
\begin{proof}
 Por la Proposición~\ref{proposition:cat-aplicacao-invariante-homotopico}, tenemos que $\mathrm{cat}(1_X)=\mathrm{cat}(g\circ f)$ y $\mathrm{cat}(1_Y)=\mathrm{cat}(f\circ g)$. Por la Proposición~\ref{proposition:cat-composta-equivalencia-homotopica}, sigue que $\mathrm{cat}(g\circ f)=\mathrm{cat}(f)$ (ya que $g$ es una equivalencia homotópica) y $\mathrm{cat}(g\circ f)=\mathrm{cat}(g)$ (ya que $f$ es una equivalencia homotópica). Análogamente, nuevamente por la Proposición~\ref{proposition:cat-composta-equivalencia-homotopica}, sigue que $\mathrm{cat}(f\circ g)=\mathrm{cat}(f)$ (ya que $g$ es una equivalencia homotópica). Por lo tanto, $\mathrm{cat}(f)=\mathrm{cat}(g)=\mathrm{cat}(X)=\mathrm{cat}(Y)$. 
\end{proof}

El siguiente ejemplo muestra que $\mathrm{cat}(X)=\mathrm{cat}\left(\mathrm{Map}([0,1],X)\right)$.

\begin{example}
    Sea $X$ un espacio topológico. Considere las aplicaciones $\psi:X\to\mathrm{Map}([0,1],X),~\psi(x)=\overline{x}$, donde $\overline{x}$ es el camino constante en $x$, y la aplicación $\varphi:\mathrm{Map}([0,1],X)\to X,~\varphi(\alpha)=\alpha(0)$. Note que, $\varphi\circ\psi=1_X$ y $\psi\circ\varphi\simeq 1_{\mathrm{Map}([0,1],X)}$ (de hecho, podemos considerar la homotopía $H:\mathrm{Map}([0,1],X)\times [0,1]\to \mathrm{Map}([0,1],X),~H(\alpha,t)(s)=\alpha(ts)$). Así, $\psi$ es una equivalencia homotópica con inversa homotópica $\varphi$. Luego, aplicando la Proposición~\ref{proposition:cat-map-inversa-homotopica}, obtenemos que $\mathrm{cat}(\varphi)=\mathrm{cat}(\psi)=\mathrm{cat}(X)=\mathrm{cat}\left(\mathrm{Map}([0,1],X)\right)$.  
\end{example}

De \cite[Definition 2.1, p. 268]{berstein1962} recordemos la noción de categoría seccional o género de una aplicación. Inicialmente este invariante numérico fue estudiado por Schwarz en \cite{schwarz1966} para fibraciones.

\begin{definition}\label{definition:secat-def}
   Sea $f:X\to Y$ una aplicación continua. La \textit{categoría seccional} de $f$, denotada por $\mathrm{secat}(f)$, es el menor entero positivo $k$ tal que existen $U_1,\ldots,U_k$ abiertos de $Y$ tales que $Y=U_1\cup\cdots\cup U_k$ y para cada $U_j$ existe una aplicación continua $s_j:U_j\to X$ con $f\circ s_j\simeq incl_{U_j}$, donde $incl_{U_j}:U_j\hookrightarrow Y$ es la aplicación inclusión.   
\end{definition}

Note que $\mathrm{secat}(f)=1$ si, y solamente si, existe una aplicación continua $s:Y\to X$ tal que $f\circ s\simeq 1_Y$. Además, secat es un invariante homotópico, o sea, si $f\simeq g$ entonces $\mathrm{secat}(f)=\mathrm{secat}(g)$ (ver \cite[p. 268]{berstein1962}). Adicionalmente, si $Y$ es conexo por caminos, entonces $\mathrm{secat}(f)\leq\mathrm{cat}(Y)$ (ver \cite[p. 268]{berstein1962}). 

\medskip El siguiente resultado muestra que la categoría de una aplicación junto con su categoría seccional superan a la categoría de su contradominio.

\begin{theorem}\label{theorem:cat-secat-maior-cat}
 Sea $f:X\to Y$ una aplicación continua. Tenemos \[\mathrm{cat}(f)\cdot\mathrm{secat}(f)\geq\mathrm{cat}(Y).\]   
\end{theorem}
\begin{proof}
  Sean $\mathrm{cat}(f)=n$ y $\mathrm{secat}(f)=m$. Consideremos $X=U_1\cup\cdots\cup U_n$ con cada $U_j$ abierto de $X$ tal que cada restricción $f_{| U_j}:U_j\to Y$ es nulo homotópica. También, consideremos $Y=V_1\cup\cdots\cup V_m$ con cada $V_i$ abierto de $Y$ tal que existe una aplicación continua $s_i:V_i\to X$ con $f\circ s_i\simeq incl_{V_i}$. Para cada $i\in \{1,\ldots,m\}$ y $j\in \{1,\ldots,n\}$, defina $V_{j,i}=s_i^{-1}(U_j)$. Note que cada $V_{j,i}$ es abierto de $Y$ y $Y=\bigcup_{j=1,i=1}^{n,m} V_{j,i}$. Además, \begin{align*}
      incl_{V_{j,i}}&=incl_{V_i}\circ incl_{V_{j,i}}^V& \mbox{ ($incl_{V_{j,i}}^{V_i}:V_{j,i}\hookrightarrow V_i$)}\\
      &\simeq (f\circ s_i)\circ incl_{V_{j,i}}^{V_i}\\
      &=f\circ (s_i\circ incl_{V_{j,i}}^{V_i})\\
      &=f\circ (incl_{U_j}\circ (s_i)_{|})& \mbox{($(s_i)_{|}:V_{j,i}\to U_j,~(s_i)_{|}(x)=(s_i)(x)$)}\\
      &=(f\circ incl_{U_j})\circ (s_i)_{|}\\
      &=f_{| U_j}\circ (s_i)_{|}.
  \end{align*}   Así, cada $incl_{V_{j,i}}$ es nulo homotópica, ya que $f_{| U_j}$ es nulo homotópica. Luego, $\mathrm{cat}(Y)\leq nm=\mathrm{cat}(f)\cdot\mathrm{secat}(f)$.
\end{proof}

Teorema~\ref{theorem:cat-secat-maior-cat} junto la Proposición~\ref{proposition:cat-catLS} implican la siguiente afirmación.

\begin{proposition}\label{proposition:secat-1-cat-igual-catLS}
  Sea $f:X\to Y$ una aplicación continua. Si $\mathrm{secat}(f)=1$ entonces $\mathrm{cat}(f)=\mathrm{cat}(Y)$.   
\end{proposition}

Note que, la vuelta de la Proposición~\ref{proposition:secat-1-cat-igual-catLS} no vale.

\medskip Recordemos que $\mathrm{Conf}(X,2)=\{(x_1,x_2)\in X\times X:~x_1\neq x_2\}$ denota el espacio de configuraciones ordenado de dos puntos en $X$. Consideremos $\pi^X:\mathrm{Conf}(X,2)\to X,~\pi^X(x_1,x_2)=x_1$, la proyección en la primera coordenada. Note que $\mathrm{cat}(\pi^X)=1$ para todo $X$ espacio topológico contráctil. Por otro lado, tenemos el siguiente resultado.

\begin{proposition}\label{proposition:cat-proyection-section}
Si $\pi^X$ admite una sección continua entonces $\mathrm{cat}(\pi^X)=\mathrm{cat}(X)$. 
\end{proposition}
\begin{proof}
Aplicar la Proposición~\ref{proposition:secat-1-cat-igual-catLS}. 
\end{proof}

\begin{example}
    \noindent\begin{enumerate}
        \item[(1)] La aplicación $s:S^n\to \mathrm{Conf}(S^n,2),~s(x)=(x,-x)$, es una sección continua de $\pi^{S^n}:\mathrm{Conf}(S^n,2)\to S^n$. Luego, por la Proposición~\ref{proposition:cat-proyection-section}, tenemos que $\mathrm{cat}\left(\pi^{S^n}\right)=2$, ya que $\mathrm{cat}(S^n)=2$.
        \item[(2)] Sea $G$ un grupo topológico con por lo menos dos elementos. Tenemos que $\mathrm{cat}(\pi^G)=\mathrm{cat}(G)$. De hecho, como $G$ tiene al menos dos elementos, podemos tomar un elemento $h\in G\setminus\{e\}$, donde $e$ es el elemento neutro del grupo $G$. Luego, la aplicación $s:G\to \mathrm{Conf}(G,2),~s(g)=(g,gh)$, es una sección continua de $\pi^G$. Nuevamente, por la Proposición~\ref{proposition:cat-proyection-section}, tenemos que $\mathrm{cat}(\pi^G)=\mathrm{cat}(G)$.   
    \end{enumerate}
\end{example}

No conocemos un ejemplo de espacio topológico $X$ tal que $\mathrm{cat}(\pi^X)<\mathrm{cat}(X)$. Así, planteamos el siguiente problema.

\begin{problem}
  Sea $X$ un espacio topológico. ¿Es verdad que $\mathrm{cat}(\pi^X)=\mathrm{cat}(X)$?  
\end{problem}

En particular, planteamos el siguiente problema.

\begin{problem}
  ¿Es verdad que $\mathrm{cat}(\pi^{\mathbb{R}P^{2n}})=\mathrm{cat}(\mathbb{R}P^{2n})$?  
\end{problem}

Más generalmente, planteamos el siguiente problema.

\begin{problem}
  Sea $X$ un espacio topológico. Calcular $\mathrm{cat}(\pi^X_{k,r})$, donde $\pi^X_{k,r}:\mathrm{Conf}(X,k)\to\mathrm{Conf}(X,r),~\pi^X_{k,r}(x_1,\ldots,x_k)=(x_1,\ldots,x_r)$, es la proyección en las $r$ primeras coordenadas ($r\leq k$). Además, $\mathrm{Conf}(X,k)=\{(x_1,\ldots,x_k)\in X^k:~x_i\neq x_j \text{ para todo $i\neq j$}\}$ denota el \textit{espacio de configuraciones ordenado}\index{Espacio de configuraciones ordenado} de $k$ puntos en $X$.   
\end{problem}

Ver \cite{zapata2020} para algunos resultados parciales de $\mathrm{secat}(\pi^X_{k,r})$.

%%%%%%%%%%%%%%%%%%%%%%%%%%%%%%%%%%%%%%%%%%%%%%%%%%%%%%%%%%%%%%%%%%
%%%%%%%%%%%%%%%%%%%%%%%%%%%%%%%%%%%%%%%%%%%%%%%%%%%%%%%%%%%%%%%%%%%
\section{Aplicaciones en análisis no lineal}\label{chap:aplicaciones-analise}
En este capítulo usaremos la teoría de categoría de una aplicación para estudiar la existencia de soluciones de ecuaciones no lineales. En la Sección~\ref{sec:ecuaciones-no-lineales} presentamos una conexión entre la existencia de soluciones de una ecuación no lineal y homotopía. En la Sección~\ref{sec:ejemplos-concretos} presentamos una serie de ejemplos del uso de categoría en la existencia de soluciones de ecuaciones no lineales. 

\subsection{Ecuaciones no lineales}\label{sec:ecuaciones-no-lineales}
En esta sección se presenta el problema clásico de análisis no lineal, en particular de la existencia de soluciones, el cual pretendemos resolver usando la teoría de categoría de una aplicación. Una conexión entre la existencia de soluciones de una ecuación no lineal y la categoría de una aplicación está dada en el Teorema~\ref{theorem:extension-cero-cat-2}. En particular, el  Corolario~\ref{corollary:cat-2-solucion} presenta una condición en términos de categoría para la existencia de soluciones de una ecuación no lineal. Proposición~\ref{proposition:centro-radio-cualquer} dice que podemos considerar cualquier disco cerrado de cualquier radio y centrado en cualquier punto. Una condición en términos de categoría seccional para la existencia de soluciones está dada en la Proposición~\ref{proposition:secat-existencia-soluciones}.

\medskip Un problema clásico en análisis es resolver ecuaciones no lineales de la forma \begin{equation}\label{equation}
      F(x)=0,
  \end{equation} donde $F:D^n\to \mathbb{R}^m$ es una aplicación continua del disco unitario cerrado $D^n\subset\mathbb{R}^n$ en $\mathbb{R}^m$. Note que, si existe $x\in \partial D^n=S^{n-1}$ tal que $F(x)=0$, tenemos inmediatamente que la ecuación~(\ref{equation}) admite solución. Así que, en adelante vamos a suponer que $F(x)\neq 0$, para cualquier $x\in \partial D^n=S^{n-1}$ y consideraremos $F_\mid:S^{n-1}\to\mathbb{R}^m-\{0\}$ como la aplicación restricción.

\medskip Para fines de esta sección presentamos la siguiente definición.
  
  \begin{definition}
      Sean $X\subset Z$ y $Y\subset W$ espacios topológicos y $f:X\to Y$ una aplicación continua. Diremos que una aplicación continua $\varphi:Z\to W$ es una \textit{extensión}\index{Extensión} de $f$ cuando $\varphi(x)=f(x)$ para todo $x\in X$.
  \end{definition} 

  \begin{example}\label{F-extension}
  Sea $F:D^n\to \mathbb{R}^m$ una aplicación continua tal que $F(x)\neq 0$, para cualquier $x\in S^{n-1}$ y considere $F_\mid:S^{n-1}\to\mathbb{R}^m-\{0\}$ la aplicación restricción. Tenemos que $F$ es una extensión de $F_\mid$.   
  \end{example}
  
  \medskip En \cite[Theorem 1.1.1, p. 1]{nirenberg1974} se presentó la siguiente afirmación y su demostración fue dejada como ejercicio. Esta afirmación muestra una conexión entre la existencia de soluciones de ecuaciones y homotopía. En esta sección presentamos una demostración de esta afirmación. Sea $F:D^n\to \mathbb{R}^m$ una aplicación continua tal que $F(x)\neq 0$, para cualquier $x\in S^{n-1}$ y considere $F_\mid:S^{n-1}\to\mathbb{R}^m-\{0\}$ la aplicación restricción.

  \begin{theorem}\label{theorem:extension-cero-nulo-homo}
     Cada extensión $\varphi:D^n\to \mathbb{R}^m$ de $F_\mid$ admite un cero, o sea, existe $x\in D^n$ tal que $\varphi(x)=0$, si y solamente si la restricción $F_\mid$ no es nulo homotópica.   
  \end{theorem}
  \begin{proof}
   $(\Rightarrow)$ Por contradicción. Supongamos que $F_\mid$ es nulo homotópica y considere una nulo homotopía $H:S^{n-1}\times [0,1]\to \mathbb{R}^{m}-\{0\}$ con $H_0=F_\mid$ y $H_1=\overline{c}$, para alguna constante $c\in \mathbb{R}^{m}-\{0\}$. Defina la aplicación $\varphi:D^n\to \mathbb{R}^m$ por:
   \[\varphi(x)=\begin{cases}
 c, &\hbox{ si $0\leq\parallel x\parallel\leq 1/2$;}\\
 H\left(\dfrac{x}{\parallel x\parallel},2-2\parallel x\parallel\right), &\hbox{ si $1/2\leq\parallel x\parallel\leq 1$.} 
\end{cases}
\] Note que $\varphi$ es una extensión de $F_\mid$ y $\varphi(x)\neq 0$ for any $x\in D^n$. Lo cual es una contradicción. Por lo tanto, $F_\mid$ no es nulo homotópica.

$(\Leftarrow)$ Por contradicción. Supongamos que existe una extensión $\varphi:D^n\to \mathbb{R}^m$ de $F_\mid$ tal que $\varphi(x)\neq 0$, para cualquier $x\in D^n$. Considere la homotopia $H:S^{n-1}\times [0,1]\to \mathbb{R}^m-\{0\}$ dada por 
\[H(x,t)=\varphi\left((1-t)x\right).
\] Note que $H$ cumple que $H_0=F_\mid$ y $H_1=\overline{c}$, donde $c=\varphi\left(0\right)\in \mathbb{R}^m-\{0\}$. Luego, $F_\mid$ es nulo homotópica. Lo cual es una contradicción. Por lo tanto, cada extensión $\varphi:D^n\to \mathbb{R}^m$ de $F_\mid$ admite un cero.
  \end{proof}

Teorema~\ref{theorem:extension-cero-nulo-homo} implica el siguiente resultado que muestra una condición homotópica para la existencia de solución de una ecuación.  

  \begin{corollary}\label{corollary:no-nulo-solucion}
 Sea $F:D^n\to \mathbb{R}^m$ una aplicación continua tal que $F(x)\neq 0$, para cualquier $x\in S^{n-1}$ y considere $F_\mid:S^{n-1}\to\mathbb{R}^m-\{0\}$ la aplicación restricción. Si la restricción $F_\mid$ no es nulo homotópica entonces la ecuación $F(x)=0$ admite solución.   
  \end{corollary}
  \begin{proof}
      Aplicando el Teorema~\ref{theorem:extension-cero-nulo-homo} junto con el Ejemplo~\ref{F-extension}, obtenemos que existe $x\in D^n$ ($x\notin S^{n-1}$) tal que $F(x)=0$.
  \end{proof}

  Usando la Proposición~\ref{proposition:cat-no-nulo-2} podemos expresar el Teorema~\ref{theorem:extension-cero-nulo-homo} en términos de categoría.

  \begin{theorem}\label{theorem:extension-cero-cat-2}
 Sea $F:D^n\to \mathbb{R}^m$ una aplicación continua tal que $F(x)\neq 0$, para cualquier $x\in S^{n-1}$ y considere $F_\mid:S^{n-1}\to\mathbb{R}^m-\{0\}$ la aplicación restricción. Cada extensión $\varphi:D^n\to \mathbb{R}^m$ de $F_\mid$ admite un cero, o sea, existe $x\in D^n$ tal que $\varphi(x)=0$, si y solamente si $\mathrm{cat}(F_\mid)=2$.   
  \end{theorem}

  En particular, el Corolario~\ref{corollary:no-nulo-solucion} en términos de categoría queda dado de la siguiente manera.

   \begin{corollary}\label{corollary:cat-2-solucion}
 Sea $F:D^n\to \mathbb{R}^m$ una aplicación continua tal que $F(x)\neq 0$, para cualquier $x\in S^{n-1}$ y considere $F_\mid:S^{n-1}\to\mathbb{R}^m-\{0\}$ la aplicación restricción. Si $\mathrm{cat}(F_\mid)=2$  entonces la ecuación $F(x)=0$ admite solución.   
  \end{corollary}

  Corolario~\ref{corollary:cat-2-solucion} implica el siguiente resultado. El cual dice que podemos considerar cualquier disco cerrado de cualquier radio y centrado en cualquier punto. Para $x_0\in\mathbb{R}^n$ y $r>0$, el \textit{disco cerrado} en $\mathbb{R}^n$ de \textit{radio} $r$ y \textit{centro} $x_0$ está dado por $D^n_r(x_0)=\{x\in\mathbb{R}^n:~||x-x_0||\leq r\}$ y la \textit{esfera} en $\mathbb{R}^n$ de \textit{radio} $r$ y \textit{centro} $x_0$ está dada por $S^{n-1}_r(x_0)=\{x\in\mathbb{R}^n:~||x-x_0||= r\}$. 

  \begin{proposition}\label{proposition:centro-radio-cualquer}
  Sea $F:D^n_r(x_0)\to \mathbb{R}^m$ una aplicación continua tal que $F(x)\neq 0$, para cualquier $x\in S^{n-1}_r(x_0)$ y considere $F_\mid:S^{n-1}_r(x_0)\to\mathbb{R}^m-\{0\}$ la aplicación restricción. Si $\mathrm{cat}(F_\mid)=2$  entonces la ecuación $F(x)=0$ admite solución.      
  \end{proposition}
  \begin{proof}
    Considere el homeomorfismo $\varphi:D^n_r(x_0)\to D^n,~\varphi(x)=\dfrac{1}{r}(x-x_0)$, cuya inversa es la aplicación $\psi:D^n\to D^n_r(x_0),~\psi(y)=ry+x_0$. Note que $\varphi(x)\in S^{n-1}$, para todo $x\in S^{n-1}_r(x_0)$, y $\psi(y)\in S^{n-1}_r(x_0)$, para todo $y\in S^{n-1}$. Considere la aplicación $G:D^n\to \mathbb{R}^m,~G(y)=F(\psi(y))=F(ry+x_0)$. Como $F(x)\neq 0$, para cualquier $x\in S^{n-1}_r(x_0)$, tenemos que $G(y)\neq 0$, para cualquier $y\in S^{n-1}$. Además,la aplicación restricción $G_{|}:S^{n-1}\to \mathbb{R}^m\setminus\{0\}$ cumple $G_{|}=F_{|}\circ \psi_{| S^{n-1}}$, donde $\psi_{| S^{n-1}}:S^{n-1}\to S^{n-1}_r(x_))$ es la aplicación restricción. Luego, por la Proposición~\ref{proposition:cat-map-inversa-homotopica}, se tiene que $\mathrm{cat}(G_|)=\mathrm{cat}(F_|)=2$. Así, por el Corolario~\ref{corollary:cat-2-solucion} tenemos que la ecuación $G(y)=0$ tiene solución, o sea, existe $y'\in D^{n}$ tal que $G(y')=0$. Luego, el punto $x'=\psi(y')=ry'+x_0\in D^n_r(x_0)$ es una solución de  la ecuación $F(x)=0$. 
  \end{proof}

  Note que, el regreso del Corolario~\ref{corollary:cat-2-solucion} no vale.% (ver Ejercicio~\ref{exercise:cat-map-1-admite-solucion}).

  \medskip La demostración del Teorema~\ref{theorem:extension-cero-nulo-homo} (equivalentemente Teorema~\ref{theorem:extension-cero-cat-2}) fue realizada por contradicción. En particular, la existencia de solución no es constructiva. Así, un problema interesante es el siguiente:

  \begin{problem}
   Sea $F:D^n\to \mathbb{R}^m$ una aplicación continua tal que $F(x)\neq 0$, para cualquier $x\in S^{n-1}$ y considere $F_\mid:S^{n-1}\to\mathbb{R}^m-\{0\}$ la aplicación restricción. Suponga que $\mathrm{cat}(F_\mid)=2$ y $\varphi:D^n\to \mathbb{R}^m$ es una extensión de $F_\mid$. Encuentre un método para construir un cero de $\varphi$, o sea, para encontrar un $x\in D^n$ tal que $\varphi(x)=0$.   
  \end{problem}

La siguiente observación muestra una limitación del Corolario~\ref{corollary:cat-2-solucion}.

  \begin{remark}\label{remark:deficiencia}
     Para $n<m$, por el \textit{Teorema de aproximación celular} (ver \cite[Theorem 4.8, p. 349]{hatcher2002}), se tiene que cualquier aplicación continua $f:S^{n-1}\to\mathbb{R}^{m-1}$ es nulo homotópica (ver \cite[Corollary 4.9, p. 349]{hatcher2002}), o sea, su categoría $\mathrm{cat}(f)=1$. En este caso el Corolario~\ref{corollary:cat-2-solucion} no será posible usarlo.  
  \end{remark}

  Para subsanar la deficiencia del Corolario~\ref{corollary:cat-2-solucion} presentada en la Observación~\ref{remark:deficiencia} se puede reemplazar el conjunto unitario $\{0\}$ por cualquier subespacio  no vacío $C\subset \mathbb{R}^m$. % (ver Ejercicio~\ref{exercise:esfera-subespacio-cualquier})

  \medskip Sean $C\subset \mathbb{R}^m$ un subconjunto no vacío y $F:D^n\to \mathbb{R}^m$ una aplicación continua tal que $F(x)\not\in C$, para cualquier $x\in S^{n-1}$ y considere $F_\mid:S^{n-1}\to\mathbb{R}^m-C$ la aplicación restricción. Suponga que $\mathrm{cat}(F_\mid)=2$. Se tiene que la ecuación $F(x)\in C$ admite solución, o sea, existe un $x_0\in D^n$ tal que $F(x_0)\in C$.

\medskip Por otro lado, en el Corolario~\ref{corollary:cat-2-solucion}, si cambiamos el $n$-disco $D^n$ por el disco infinito dimensional $D^\infty$ de $\mathbb{R}^\infty$ y la $(n-1)$-esfera $S^{n-1}$ por la esfera infinito-dimensional $S^\infty$, se tiene que cualquier aplicación continua $f:S^\infty\to \mathbb{R}^m\setminus\{0\}$ es nulo homotópica, o sea, $\text{cat}(f)=1$. Ya que, $S^\infty$ es contráctil, ver \cite[Example 1B.3, pg. 88]{hatcher2002} o \cite[Theorem 11.1.3, pg. 332]{aguilar2002}. Así, una versión análoga del Corolario~\ref{corollary:cat-2-solucion} en espacios de dimensión infinita no será posible usarlo. Un problema natural es extender el Corolario~\ref{corollary:cat-2-solucion} para espacio de dimensión infinita, de manera análoga como es extendida la teoría del grado topológico para espacios de dimensión infinita (ver \cite[Chapter 2]{nirenberg1974}).

  \begin{problem}
    Extienda el Corolario~\ref{corollary:cat-2-solucion} para espacios de dimensión infinita.  
  \end{problem}

  Usando la Proposición~\ref{proposition:secat-1-cat-igual-catLS} junto con la Proposición~\ref{proposition:centro-radio-cualquer} obtenemos una condición en términos de categoría seccional para la existencia de soluciones.

  \begin{proposition}\label{proposition:secat-existencia-soluciones}
      Sea $F:D^n_r(x_0)\to \mathbb{R}^m$ una aplicación continua tal que $F(x)\neq 0$, para cualquier $x\in S^{n-1}_r(x_0)$ y considere $F_\mid:S^{n-1}_r(x_0)\to\mathbb{R}^m-\{0\}$ la aplicación restricción. Si $\mathrm{secat}(F_\mid)=1$  entonces la ecuación $F(x)=0$ admite solución.  
  \end{proposition}
  \begin{proof}
    Como $\mathrm{secat}(F_\mid)=1$, por la Proposición~\ref{proposition:secat-1-cat-igual-catLS}, tenemos que $\mathrm{cat}(F_\mid)=\mathrm{cat}(\mathbb{R}^m-\{0\})=2$. Así, aplicando la Proposición~\ref{proposition:centro-radio-cualquer}, sigue que la ecuación $F(x)=0$ admite solución. 
  \end{proof}

%%%%%%%%%%%%%%%%%%%%%%%%%%%%%%%%%%%%%%%%%%%%%%%%%%%%%%%%%%%%%%%%%%%%%%%%%%%%%%%%%%%%%%%%%%%%%%%%%%%%%%%%%%%%%%%%%%%%%%%
  \subsection{Ejemplos}\label{sec:ejemplos-concretos}
  En esta sección presentamos una serie de ejemplos del uso de categoría en la existencia de ecuaciones no lineales.
  
\begin{example}\label{opossite}
    Sea $F:D^n\to \mathbb{R}^n$ una aplicación continua tal que $F(x)$ nunca apunta en dirección opuesta a $x$ para cualquier $x\in S^{n-1}$, i.e., $F(x)\neq \lambda x$, para todo $\lambda< 0$, para todo $x\in S^{n-1}$. Entonces la ecuación \begin{align}\label{nver-opposite-point}
        F(x)&=0,
    \end{align} tiene solución. De hecho, podemos considerar la homotopía $H:S^{n-1}\times [0,1]\to \mathbb{R}^n-\{0\}$ dada por \[H(x,t)=(1-t)x+tF(x).\] 
    Note que, $H_0=incl_{S^{n-1}}$ y $H_1=F_\mid$, donde $incl_{S^{n-1}}:S^{n-1}\hookrightarrow \mathbb{R}^n-\{0\}$ es la aplicación inclusión. Del Ítem (3) del Ejemplo~\ref{examples-retracts} tenemos que la inclusión $incl_{S^{n-1}}:S^{n-1}\hookrightarrow \mathbb{R}^n-\{0\}$ es una equivalencia homotópica, luego, por la Proposición~\ref{proposition:cat-map-inversa-homotopica}, se tiene que $\mathrm{cat}(incl_{S^{n-1}})=\mathrm{cat}(S^{n-1})=2$ (para la última igualdad ver Ítem (2) del Ejemplo~\ref{exam:catLS}). Entonces \begin{align*}
        \mathrm{cat}(F_\mid)&=\mathrm{cat}(incl_{S^{n-1}}) & \mbox{(sigue de la Proposición~\ref{proposition:cat-aplicacao-invariante-homotopico})}\\
        %&=\mathrm{cat}(S^{n-1})\\
        &=2.
    \end{align*} Luego, por el Corolario~\ref{corollary:cat-2-solucion}, la ecuación~(\ref{nver-opposite-point}) tiene solución.
\end{example}

El siguiente ejemplo dice que podemos obtener una versión del Ejemplo~\ref{opossite} para un disco de cualquier radio y centrado en el origen. 

\begin{example}\label{opossite-radio-centro-cualquier}
    Sean $r>0$ y $F:D^n_r(0)\to \mathbb{R}^n$ una aplicación continua tal que $F(x)$ nunca apunta en dirección opuesta a $x$ para cualquier $x\in S^{n-1}_r(0)$, i.e., $F(x)\neq \lambda x$, para todo $\lambda< 0$, para todo $x\in S^{n-1}_r(0)$. Entonces la ecuación \begin{align*}
        F(x)&=0,
    \end{align*} tiene solución. De hecho, podemos considerar el homeomorfismo $\varphi:D^n_r(0)\to D^n,~\varphi(x)=\dfrac{1}{r}(x)$, cuya inversa es la aplicación $\psi:D^n\to D^n_r(0),~\psi(y)=ry$. Note que $\varphi(x)\in S^{n-1}$, para todo $x\in S^{n-1}_r(0)$, y $\psi(y)\in S^{n-1}_r(0)$, para todo $y\in S^{n-1}$. Considere la aplicación $G:D^n\to \mathbb{R}^m,~G(y)=F(\psi(y))=F(ry)$. Como $F(x)\neq \lambda x$, para todo $\lambda< 0$, para todo $x\in S^{n-1}_r(0)$, tenemos que $G(y)\neq \lambda y$, para todo $\lambda< 0$, para todo $y\in S^{n-1}$. Luego, aplicando el Ejemplo~\ref{opossite} a $G$, tenemos que existe $y_0\in D^n$ tal que $G(y)=0$. Así, $x=ry\in D^n_r(0)$ es una solución para la ecuación $F(x)=0$.
\end{example}

\medskip Ejemplo~\ref{opossite-radio-centro-cualquier} implica el siguiente resultado.

\begin{example}
    Sea $F:\mathbb{R}^n\to \mathbb{R}^n$ una aplicación continua tal que $\dfrac{\langle F(x),x\rangle}{\parallel x\parallel}\to +\infty$ uniformemente cuando $\parallel x\parallel\to +\infty$. Entonces la ecuación \begin{align}\label{uniformy}
        F(x)&=0,
    \end{align} tiene solución. De hecho, sea $R>0$ tal que $\langle F(x);x\rangle\geq 0$, para cualquier $\parallel x\parallel=R$. Si $F(x_0)=0$ para algún $\parallel x_0\parallel=R$, no hay nada que hacer. Supongamos que, $F(x)\neq 0$, para cualquier $\parallel x\parallel=R$. Entonces $F(x)$ nunca apunta en dirección opuesta a $x$, para cualquier $\parallel x\parallel=R$. Luego, por el Ejemplo~\ref{opossite-radio-centro-cualquier}, tenemos que la ecuación~(\ref{uniformy}) tiene solución.
\end{example}

%%%%%%%%%%%%%%%%%%%%%%%%%%%%%%%%%%%%%%%%%%%%%%%%%%%%%%%%%%%%%%%%%%%%%%
%%%%%%%%%%%%%%%%%%%%%%%%%%%%%%%%%%%%%%%%%%%%%%%%%%%%%%%%%%%%%%%%%%%%%%%%%%
\section{Conclusiones}
Una técnica topológica, que existe en la literatura, para la existencia de ecuaciones no lineales es la teoría del grado topológico. En este trabajo, usamos la teoría de categoría de una aplicación para resolver el problema de existencia de soluciones de ecuaciones no lineales. Esta teoría, como mostramos en este trabajo, da una técnica alternativa para estudiar ecuaciones no lineales.

%%%%%%%%%%%%%%%%%%%%%%%%%%%%%%%%%%%%%%%%%%%%%%%%%%%%%%%%%%%%%%%%%%%%%%%%%%%%%%%%%%%%%%%%%%%%%%%%%
%%%%%%%%%%%%%%%%%%%%%%%%%%%%%%%%%%%%%%%%%%%%%%%%%%%%%%%%%%%%%%%%%%%%%%%%%%%%%%%%%%%%%%%%%%%%%%%%%%%
\section{Agradecimientos}
Estas Notas fueron preparadas como material académico para ser usadas durante la visita científica realizada por el autor en el Departamento de Matemáticas de la Pontificia Universidad Católica del Perú (PUCP) en Lima, Perú, entre el período del 17 al 24 de Julio del 2023. El autor desea agradecer la subvención\#2022/16695-7, S\~{a}o Paulo Research Foundation (FAPESP) por el apoyo financiero. Adicionalmente desea agradecer al Departamento de Matemáticas de la Pontificia Universidad Católica del Perú (PUCP) por todo el apoyo brindado durante la visita científica. En particular, el autor considerablemente agradece a los profesores Roland Rabanal, Percy Braulio Fernandez Sanchez y Jorge Chávez.
%%%%%%%%%%%%%%%%%%%%%%%%%%%%%%%%%%%%%%%%%%%%%%%%%%%%%%%%%%%%%%%%%
%%%%%%%%%%%%%%%%%%%%%%%%%%%%%%%%%%%%%%%%%%%%%%%%%%%%%%%%%%%%%%%%%

\bibliographystyle{plain}

%\printindex
\end{document}